\documentclass[11pt]{article}
\usepackage{a4}
\usepackage{amsmath,amsfonts,amssymb,amscd,amsthm}
\usepackage{graphicx}

\usepackage{subfigure}
\usepackage{pict2e}
\usepackage{color}
\usepackage[percent]{overpic}
\usepackage[colorlinks=true,linkcolor=blue,urlcolor=blue]{hyperref}
\usepackage{pbox}
\usepackage{float} 
\usepackage{algorithm}
\usepackage{algorithmic}

\newtheorem{theorem}{Theorem}

\newtheorem{proposition}[theorem]{Proposition}
\newtheorem{definition}[theorem]{Definition}
\newtheorem{lemma}[theorem]{Lemma}
\newtheorem{corollary}[theorem]{Corollary}
\newtheorem{remark}[theorem]{Remark}
\newtheorem{example}{Example}

\DeclareMathOperator{\R}{\mathbb{R}}

\begin{document}
\title{hp-version time domain boundary elements for the wave equation on quasi-uniform meshes}
\author{ Heiko Gimperlein\thanks{Maxwell Institute for Mathematical Sciences and Department of Mathematics, Heriot--Watt University, Edinburgh, EH14 4AS, United Kingdom, email: h.gimperlein@hw.ac.uk.}  \thanks{Institute for Mathematics, University of Paderborn, Warburger Str.~100, 33098 Paderborn, Germany.}  \and Ceyhun \"{O}zdemir\thanks{Institute of Applied Mathematics, Leibniz University Hannover, 30167 Hannover, Germany. \newline H.~G.~acknowledges support by ERC Advanced Grant HARG 268105 and the EPSRC Impact Acceleration Account.} \and  David Stark${}^\ast$ \and Ernst P.~Stephan${}^\ddagger$}\date{}

\providecommand{\keywords}[1]{{\textit{Key words:}} #1}

\maketitle \vskip -0.5cm

\begin{abstract}
\noindent Solutions to the wave equation in the exterior of a polyhedral domain or a screen in $\mathbb{R}^3$ exhibit singular behavior from the edges and corners.  We present quasi-optimal $hp$-explicit estimates for the approximation of the Dirichlet and Neumann traces of these solutions \textcolor{blue}{for uniform time steps and (globally)} quasi-uniform meshes on the boundary. The results are applied to an $hp$-version of the time domain boundary element method. Numerical examples confirm the theoretical results for the Dirichlet problem both for screens and polyhedral domains. \end{abstract}

\section{\textcolor{black}{Introduction}}

This article initiates the study of high-order boundary elements in the time domain. For elliptic problems,  $p$- and $hp$-versions of the finite element method give rise to fast approximations of both smooth solutions and geometric singularities.  These methods converge to the solution by increasing the polynomial degree $p$ of the elements, possibly in combination with reducing the mesh size $h$ of the quasi-uniform mesh. They were first investigated in the group of Babuska \cite{babuskap2, babuskap, dorr1, dorr2}.  See \cite{schwab} for a comprehensive analysis for $2d$ problems.\\

The analogous $p$- and $hp$-versions of the boundary element method go back to \cite{alarcon, stephanp, stephanhp}.  More recent optimal convergence results for boundary elements on screens and polyhedral surfaces covering $3d$ problems have been obtained, for example, in \cite{besp, heuer01, heuer2, heuer3, bh08b}. \\

Boundary element methods for time dependent problems have recently become of interest \cite{sayas}. In this article we introduce a space-time $hp$-version of the time domain boundary element method for the wave  equation \textcolor{black}{with non-homogeneous Dirichlet or Neumann boundary conditions}. To be specific, in the exterior $\Omega \subset \mathbb{R}^3$ of a polyhedral surface or screen $\Gamma$ this article considers the initial-boundary value problem
\begin{align} \label{eq:oriProblemp}
c^{-2}\partial_t^2 u(t,x) -\Delta_x u(t,x)&=0 &\quad &\text{in }\mathbb{R}^+_t \times \Omega_x\ , \\
{u}(0,x)=\partial_t u(0,x)&=0 &\quad &\text{in }  \Omega \ , \nonumber 
\end{align}
for given Dirichlet (\textcolor{black}{$u|_\Gamma = g$}) or Neumann data (\textcolor{black}{$\partial_n u|_\Gamma = g$, $n$ \textcolor{black}{outer} unit  normal vector}) on $\Gamma = \partial \Omega$. We choose units such that $c=1$.\\

\textcolor{black}{As geometric prototypes for $\Gamma$, in this article we consider flat circular and polygonal screens, which pose the greatest numerical challenges. Extensions to closed polyhedral surfaces are mentioned.} To solve \eqref{eq:oriProblemp} numerically, we reformulate it as a time dependent integral equation on $\Gamma$ for the single layer or hypersingular operator. This integral equation is approximated using Galerkin $hp$-version boundary elements, based on tensor products of piecewise polynomial functions on a quasi-uniform mesh in space and a uniform mesh in time.\\ 

Similar to $h$-version boundary elements, the approximation rate is determined by the singularities of the solution $u$ of \eqref{eq:oriProblemp} at non-smooth boundary points of the domain.  Near an edge or a corner a singular decomposition of the solution into a leading part of explicit singular functions plus smoother terms has been obtained in a series of works by Plamenevskii and collaborators \cite{kokotov, kokotov3, matyu,plamenevskii}. Their results imply that
at a fixed time $t$, the solution to the \textcolor{black}{inhomogeneous} wave equation \textcolor{black}{with homogeneous boundary conditions} admits an explicit singular expansion with exactly the same behavior as for elliptic equations. (For the latter, see \cite{dauge, petersdorff, petersdorff2}.) \\

\textcolor{black}{Using these works, we give precise asymptotic expansions of both the boundary trace of the solution $u|_\Gamma$ to the inhomogeneous boundary problem \eqref{eq:oriProblemp} and its normal derivative $\partial_n u|_\Gamma$ on $\Gamma$. We then study their approximation by piecewise polynomials of tensor-product form in space and time. Quasi-optimal convergence rates in space-time anisotropic Sobolev spaces are obtained  for the $p$- and $hp$-versions of the boundary element method in the time domain  on flat screens, with extensions to polyhedral surfaces. For the circular screen this result is the content of Theorem \ref{approxtheorem1}, while for the polygonal screen it is Theorem \ref{approxtheorem2}. }\\

\textcolor{black}{The approximation results for $u|_\Gamma$ and $\partial_n u|_\Gamma$ imply quasi-optimal approximation results for the solution to the relevant boundary integral formulations, in Corollary \ref{approxcor1} for the circular screen, respectively Corollary \ref{approxcor2} for the polygonal screen. Indeed, on the flat screen the solution $\phi$ to the hypersingular integral equation is given by the jump $\phi = \left[u\right]|_\Gamma$ across $\Gamma$ of the solution $u$ to the wave equation with Neumann boundary conditions $\partial_n u|_\Gamma = g$. Similarly, the solution $\psi$ to the single layer integral equation is given by $\psi = \left[\partial_n u\right]|_\Gamma$ in terms of the solution $u$ to the wave equation with Dirichlet boundary conditions $u|_\Gamma=f$. This result from the elliptic case \cite{epsscreen} translates verbatim into the time-dependent setting.}\\

We mention a generalization of our results to polyhedral domains in Section \ref{polhed}.  \\

Together with the a priori estimates for the time domain boundary element methods on screens \cite{HGEPSN, gimperleintyre}, our results imply convergence rates for the $p$-version Galerkin approximations which are twice those observed for the quasi-uniform $h$-method in \cite{graded}. \\

We first prove the  approximation properties  on the circular screen, without corners, and then discuss the approximation of the corner and corner-edge singularities on polygonal screens. On the square the convergence rate is determined by the singularities at the edges.   \\

Our numerical experiments in Section \ref{experimentsscr} confirm the theoretical results and exhibit the predicted convergence rate for the Dirichlet problem for the time dependent wave equation outside a square screen. 
 The convergence rate in the energy norm is doubled compared to the convergence rate of the $h$-version on a uniform mesh, as predicted. Our numerical experiments in Section \ref{experimentsico} for the wave equation outside an icosahedron similarly confirm the predicted convergence of the $p$-method.\\

Related previous work for the time independent Laplace  equation includes, in particular, the analysis of the $p$-version by Schwab and Suri \cite{schwabs} of the singularities of $u$  in polyhedral domains and their implications for the numerical approximation of the hypersingular integral  equation by boundary elements. On geometrically graded meshes the $hp$-version was studied in \cite{heuer4}, but the analysis does not yield a priori estimates on quasi-uniform meshes. Sharp estimates on piecewise flat open surfaces with quasi-uniform meshes are due to Bespalov and Heuer \cite{heuer2, heuer3, bh08b} for both the single layer and hypersingular integral  equations. See \cite{besp,heuer01} for extensions to the Lam\'{e} equation.\\

\textcolor{black}{The analysis of boundary element methods for the wave equation goes back to seminal works of Bamberger and Ha-Duong \cite{BamHa}, with significant recent refinements by Joly and Rodriguez \cite{jr}. Alternative energy-based formulations have been studied by Aimi and collaborators \cite{aimi}. For screen problems we refer to Becache and Ha-Duong \cite{b1,b2}. Higher-order methods based on non-polynomial approximation spaces have recently been considered \cite{hgs}. See \cite{costabel04, haduong,sayas} for overviews of the theory.}\\

\noindent \emph{The article is organized as follows:} 
Section \ref{faframework} recalls the boundary integral operators associated to the wave equation as well as their mapping properties between suitable space-time anisotropic Sobolev spaces. It concludes by reformulating the Dirichlet and Neumann problems for the wave equation \eqref{eq:oriProblemp} as boundary integral equations in the time domain. The following Section \ref{Discretisationp} introduces the space-time discretizations and a time domain boundary element method to solve the single layer and hypersingular integral equations. The asymptotic expansions of solutions to the wave equation and their $hp$ approximation are the content of Section \ref{sectasympt}, for circular and polygonal screens as well as for polyhedral surfaces. 
The article presents  numerical experiments both on screens and outside polyhedral domains in Section \ref{experiments}, before summarizing the conclusions in Section \ref{secconc}. \textcolor{black}{An Appendix discusses the derivation of the singular expansions for $u|_\Gamma$ and $\partial_n u|_\Gamma$, the central basis of the convergence analysis in this article.} \\

\noindent \textcolor{black}{\emph{Notation:} We write $f \lesssim g$ provided there exists a constant $C$ such that $f \leq Cg$. If the constant
$C$ is allowed to depend on a parameter $\sigma$, we write $f \lesssim_\sigma g$.}

\section{Boundary integral operators and Sobolev spaces}
\label{faframework}

Let $\Gamma$ be the boundary of a polyhedral domain in $\mathbb{R}^3$, consisting of curved, polygonal boundary faces, or an open polyhedral surface (screen) \textcolor{black}{obtained from a Lipschitz dissection of the boundary \cite[p.~99]{mclean}}.

We make the following ansatz for the solution to \eqref{eq:oriProblemp} in terms of  the single layer potential for the wave equation,
\begin{equation}\label{singlay22}
u(t,x) =\int_0^\infty\int_{\Gamma} G(t- \tau,x,y)\ {\psi}(\tau,y)\ dy \ d\tau\ .
\end{equation}
Here $G$ is a fundamental solution to the wave equation and $\psi(\tau,y)=0$ for $\tau<0$. In 3 dimensions
\begin{align*}
u(t,x) &=\frac{1}{{4} \pi} \int_\Gamma \frac{\psi(t-|x-y|,y)}{|x-y|}\ dy \ .
\end{align*}
Taking Dirichlet boundary values on $\Gamma$ of the integral \eqref{singlay22}, we obtain the single layer operator $V$:
$$V \psi(t,x)={  \int_0^\infty\int_{\Gamma} G(t- \tau,x,y)\ \psi(\tau,y)\ dy\ d\tau\, . }$$
The wave equation \eqref{eq:oriProblemp} with Dirichlet boundary conditions, $u=f$ on $\Gamma$, is equivalent to the integral equation
\begin{equation}\label{dirproblemV}
V \psi = u|_\Gamma =  f \ .
\end{equation}

\noindent In addition to $V$, also the adjoint double layer  operator $K'$, the double layer operator $K$ and the hypersingular operator $W$ on $\Gamma$ will be used:
\begin{align}\label{operators}
\nonumber K\phi(t,x)&=\int_0^\infty\int_{\Gamma} \frac{\partial G}{\partial n_y}(t- \tau,x,y)\ \phi(\tau,y)\ dy \ d\tau,\\
K' \phi(t,x)&= \int_0^\infty\int_{\Gamma} \frac{\partial G}{\partial n_x}(t- \tau,x,y)\ \phi(\tau,y)\ dy \ d\tau\, ,\\
\nonumber W \phi(t,x)&= \int_0^\infty\int_{\Gamma} \frac{\partial^2 G}{\partial n_x \partial n_y}(t- \tau,x,y)\ \phi(\tau,y)\ dy \ d\tau\ .
\end{align}

\begin{remark}\label{vanish}
On a flat screen $\Gamma \subset \mathbb{R}^2 \times \{0\}$, $\frac{\partial G}{\partial n} = 0$ and therefore $K\phi=K'\phi=0$.
\end{remark}

These operators are studied in space-time anisotropic Sobolev spaces $H_\sigma^{{r}}(\mathbb{R}^+,\widetilde{H}^{{s}}(\Gamma))$, see  \cite{HGEPSN} or \cite{haduong}. {To define the spaces for $\partial\Gamma\neq \emptyset$, extend $\Gamma$ to a closed, orientable Lipschitz manifold $\widetilde{\Gamma}$. }

Sobolev spaces of supported distributions in $\Gamma$ are defined as:
$$\widetilde{H}^{{s}}(\Gamma) = \{u\in H^{{s}}(\widetilde{\Gamma}): \mathrm{supp}\ u \subset {\overline{\Gamma}}\}\ , \quad\ {{s}} \in \mathbb{R}\ .$$
Then we set ${H}^{{s}}(\Gamma) = H^{{s}}(\widetilde{\Gamma}) / \widetilde{H}^{{s}}({\widetilde{\Gamma}\setminus\overline{\Gamma}})$. \\
To define an explicit scale of Sobolev norms, fix a partition of unity $\alpha_i$ subordinate to a covering of $\widetilde{\Gamma}$ by open sets $B_i$ and diffeomorphisms $\phi_i$ mapping each $B_i$ into the unit cube $\subset \mathbb{R}^n$. They induce a family of norms from $\mathbb{R}^d$:
\begin{equation*}
 ||u||_{{{s}},{\widetilde{\Gamma}}}=\left( \sum_{i=1}^p \int_{\mathbb{R}^n} (|\omega|^2+|\xi|^2)^{{s}}|\mathcal{F}\left\{(\alpha_i u)\circ \phi_i^{-1}\right\}(\xi)|^2 d\xi \right)^{\frac{1}{2}}\ .
\end{equation*}
Here, $\mathcal{F}$ denotes the Fourier transform. 
The norms for different $\omega \in \mathbb{C}\setminus \{0\}$ are equivalent.  The above norms induce norms on $H^{{s}}(\Gamma)$, $||u||_{{{s}},\Gamma} = \inf_{v \in \widetilde{H}^{{s}}(\widetilde{\Gamma}\setminus\overline{\Gamma})} \ ||u+v||_{{{s}},\widetilde{\Gamma}}$, and on $\widetilde{H}^{{s}}(\Gamma)$, $||u||_{{{s}},\Gamma, \ast } = ||e_+ u||_{{{s}},\widetilde{\Gamma}}$. Here, $e_+$ extends the distribution $u$ by $0$ from $\Gamma$ to $\widetilde{\Gamma}$. 

\textcolor{black}{Weighted Sobolev spaces in time for $r\in\mathbb{R}$ and $\sigma>0$: are defined as
\begin{align*}
 H^{{r}}_\sigma(\mathbb{R}^+)&=\{ u \in \mathcal{D}^{'}_{+}: e^{-\sigma t} u \in \mathcal{S}^{'}_{+}  \textrm{ and }   \|u\|_{H^r_\sigma(\mathbb{R}^+)} < \infty \}\ .
\end{align*}
Here, $\mathcal{D}^{'}_{+}$ denotes the space of distributions on $\mathbb{R}$ with support in $[0,\infty)$, and  $\mathcal{S}^{'}_{+}$ the subspace of tempered distributions. The Sobolev spaces are Hilbert spaces endowed with the norm 
\begin{align*}
\|u\|_{H^r_\sigma(\mathbb{R}^+)}&=\left(\int_{-\infty+i\sigma}^{+\infty+i\sigma}|\omega|^{2r}\ |\hat{u}(\omega)|^2\ d\omega \right)^{\frac{1}{2}}\,.
\end{align*}}

{The scale of space-time anisotropic Sobolev spaces combines the Sobolev norms in space and time:
\begin{definition}\label{sobdef}
For $r,s\in\mathbb{R}$ {and $\sigma>0$} define
\begin{align*}
 H^{{r}}_\sigma(\mathbb{R}^+,{H}^{{s}}(\Gamma))&=\{ u \in \mathcal{D}^{'}_{+}(H^{{s}}(\Gamma)): e^{-\sigma t} u \in \mathcal{S}^{'}_{+}(H^{{s}}(\Gamma))  \textrm{ and }   \|u\|_{{{r,s}},\Gamma} < \infty \}\ , \\
 H^{{r}}_\sigma(\mathbb{R}^+,\widetilde{H}^{{s}}({\Gamma}))&=\{ u \in \mathcal{D}^{'}_{+}(\widetilde{H}^{{s}}({\Gamma})): e^{-\sigma t} u \in \mathcal{S}^{'}_{+}(\widetilde{H}^{{s}}({\Gamma}))  \textrm{ and }   \|u\|_{{{r,s}},\Gamma, \ast} < \infty \}\ .
\end{align*}
$\mathcal{D}^{'}_{+}(E)$ denotes the space of distributions on $\mathbb{R}$ with support in $[0,\infty)$, taking values in $E = {H}^{{s}}({\Gamma}), \widetilde{H}^{{s}}({\Gamma})$, and  $\mathcal{S}^{'}_{+}(E)$ the subspace of tempered distributions. The Sobolev spaces are Hilbert spaces endowed with the norm 
\begin{align*}
\|u\|_{{{r,s}},\Gamma}&=\left(\int_{-\infty+i\sigma}^{+\infty+i\sigma}|\omega|^{2{{r}}}\ \|\hat{u}(\omega)\|^2_{{{s}},\Gamma}\ d\omega \right)^{\frac{1}{2}}\ ,\\
\|u\|_{{{r,s}},\Gamma,\ast}&=\left(\int_{-\infty+i\sigma}^{+\infty+i\sigma}|\omega|^{2{{r}}}\ \|\hat{u}(\omega)\|^2_{{{s}},\Gamma,\ast}\ d\omega \right)^{\frac{1}{2}}\,.
\end{align*}
\end{definition}
When $|s|\leq 1$ one can show that the spaces are independent of the choice of $\alpha_i$ and $\phi_i$. \textcolor{black}{For $s>1$, it is most convenient to define them (via the trace theorem) as the image of the boundary trace on $\Gamma$ of the spaces in the Euclidean domain $\Omega \subset \mathbb{R}^3$ \textcolor{black}{\cite{schwab}}. For the spaces on $\Omega$, the above definitions apply.}

We state the mapping properties of the boundary integral operators, \textcolor{black}{see e.g.~\cite{costabel04, haduong}, with extensions to screens considered in \cite{HGEPSN}}:
\begin{theorem}\label{mappingproperties}
The following operators are continuous for $r\in \R$, ${{\sigma>0}}$:
\begin{align*}
& V:  {H}^{r+1}_\sigma(\R^+, \tilde{H}^{-\frac{1}{2}}(\Gamma))\to  {H}^{r}_\sigma(\R^+, {H}^{\frac{1}{2}}(\Gamma)) \ ,
\\ & K':  {H}^{r+1}_\sigma(\R^+, \tilde{H}^{-\frac{1}{2}}(\Gamma))\to {H}^{r}_\sigma(\R^+, {H}^{-\frac{1}{2}}(\Gamma)) \ ,
\\ & K:  {H}^{r+1}_\sigma(\R^+, \tilde{H}^{\frac{1}{2}}(\Gamma))\to {H}^{r}_\sigma(\R^+, {H}^{\frac{1}{2}}(\Gamma)) \ ,
\\ & W:  {H}^{r+1}_\sigma(\R^+, \tilde{H}^{\frac{1}{2}}(\Gamma)))\to {H}^{r}_\sigma(\R^+, {H}^{-\frac{1}{2}}(\Gamma)) \ .
\end{align*}
\end{theorem}

\textcolor{black}{By a fundamental observation of Bamberger and Ha-Duong \cite{BamHa},} $V\partial_t$ satisfies a coercivity estimate in the norm of ${H}^{0}_\sigma(\R^+, \tilde{H}^{-\frac{1}{2}}(\Gamma))$:  $\|\psi\|^2_{0,-\frac{1}{2}, \Gamma,\ast} \lesssim_\sigma \langle V \psi, \partial_t \psi\rangle$. From the mapping properties of Theorem \ref{mappingproperties} one also has  the continuity of the bilinear form associated to $V \partial_t$ in a bigger norm: $\langle V \psi, \partial_t \psi\rangle \lesssim \|\psi\|^2_{1,-\frac{1}{2}, \Gamma,\ast}$.  Similar estimates hold for $W\partial_t$: $\|\phi\|^2_{0,\frac{1}{2}, \Gamma,\ast} \lesssim_\sigma \langle W \phi, \partial_t \phi\rangle\lesssim \|\phi\|^2_{1,\frac{1}{2}, \Gamma,\ast}$. Proofs and further information may be found in  \cite{HGEPSN, haduong}.\\

The space-time Sobolev spaces allow a precise statement and analysis of the weak formulation for the Dirichlet problem \eqref{dirproblemV}:
{Find $\psi \in H^1_{\sigma}(\mathbb{R}^+,\widetilde{H}^{-\frac{1}{2}}(\Gamma))$ such that} for all $\Psi\in H^1_{\sigma}(\mathbb{R}^+,\widetilde{H}^{-\frac{1}{2}}(\Gamma))$
\begin{equation}\label{weakform}
 \int_0^\infty\int_\Gamma (V \psi(t,{x})) \partial_t\Psi(t,{x}) \ dx\ d_\sigma t = \int_0^\infty \int_\Gamma f(t,{x}) \partial_t\Psi(t,{x}) \ dx\ d_\sigma t\ , 
\end{equation}
where $d_\sigma t = e^{-2 \sigma t}dt$.\\

For  the Neumann problem, a double layer potential ansatz for $u$:
\begin{align}\label{doubleansatz}
u(t,x)&=\int_{\mathbb{R}^+ \times \Gamma}  \frac{\partial G}{\partial n_y}(t- \tau,x,y)\ \phi(\tau,y)\ d\tau\ dy \ ,
\end{align}
with $\phi(s,y) = 0$ for $s\leq 0$ leads to the hypersingular equation
\begin{align}\label{hypersingeq}
 W \phi = \frac{\partial u}{\partial n}\Big|_\Gamma= g\ ,
\end{align}
with weak formulation:\\
Find $\phi \in H^1_\sigma(\mathbb{R}^+, \widetilde{H}^{\frac{1}{2}}(\Gamma))$ such that for all $\Phi \in H^1_\sigma(\mathbb{R}^+, \widetilde{H}^{\frac{1}{2}}(\Gamma))$ there holds:
\begin{align}\label{weakformW}
 \int_0^\infty\int_{\Gamma} (W  \phi(t,{x})) \ {\partial_t}\Phi(t,{x})  \, dx \,d_\sigma t\  = \int_0^\infty\int_{\Gamma}   g(t,{x})\ {\partial_t}\Phi(t,{x})\, dx \,d_\sigma t\ .
\end{align}

The boundary integral equations \eqref{weakform}, respectively \eqref{weakformW}, for the Dirichlet and Neumann problems are well-posed \cite{HGEPSN, gimperleintyre}:
\begin{theorem}{Let $\sigma>0$.}\\
a) Assume that $f \in H^2_{\sigma}(\mathbb{R}^+,H^{\frac{1}{2}}(\Gamma))$. Then there exists a unique solution $\psi \in H^1_{\sigma}(\mathbb{R}^+,\widetilde{H}^{-\frac{1}{2}}(\Gamma))$  of \eqref{weakform} and
\begin{equation}
\|\psi\|_{1, -\frac{1}{2}, \Gamma, \ast} \lesssim_\sigma \|f\|_{2, \frac{1}{2}, \Gamma}\ .
\end{equation}
b) Assume that $g \in H^{2}_{\sigma}(\mathbb{R}^+,H^{-\frac{1}{2}}(\Gamma))$. Then there exists a unique solution $\phi \in H^{1}_{\sigma}(\mathbb{R}^+,\widetilde{H}^{\frac{1}{2}}(\Gamma))$  of \eqref{weakformW} and
\begin{equation}
\|\phi\|_{1,\frac{1}{2}, \Gamma, \ast}\lesssim_\sigma \|g\|_{2,-\frac{1}{2}, \Gamma} \ .
\end{equation}
\end{theorem}

We finally mention some useful technical results: The first localizes estimates for fractional Sobolev norms \cite{graded}:
\begin{lemma}\label{lemma3.2}
Let $\Gamma,\, \Gamma_j\; (j=1,\dots ,N)$ be Lipschitz domains with $\overline{\Gamma} = \bigcup\limits_{j=1}^N \overline{\Gamma}_j$, $\tilde{u}\in {H^r_\sigma(\mathbb{R}^+}, \widetilde{H}^s(\Gamma)),\; u\in {H^r_\sigma(\mathbb{R}^+}, H^s(\Gamma)),\; s\in\mathbb{R}.$ Then for all $s \in [-1,1]$, $r \in \mathbb{R}$ {and $\sigma>0$}
\begin{align}
	\sum\limits_{j=1}^N \| u\|^2_{r,s,\Gamma_j} \ & \textcolor{black}{ \lesssim }\  \| u\|^2_{r,s,\Gamma} \label{3.21a}\ ,\\
	\| \tilde{u}\|^2_{r,s,\Gamma, \ast} \ & \textcolor{black}{ \lesssim }\  \sum\limits_{j=1}^N \| \tilde{u}\|^2_{r,s,\Gamma_j, \ast}\ .\label{3.21b}
\end{align}
\end{lemma}
From Lemmas 8 and 9 in \cite{graded} we recall:
\begin{lemma}\label{lemma3.3}
	Let $r\geq 0$, $0\leq s_1, s_2\leq 1$, $I_j = [0, h_j],\; u_2\in\widetilde{H}^{-s_2}(I_2)$, $u_1 \in \widetilde{H}^r_\sigma(\mathbb{R}^+, H^{-s_1}(I_1))$. Then there holds
	\begin{equation*}	
		\| u_1(t,x) u_2(y) \|_{r, -s_1 - s_2, I_1\times I_2, \ast} \lesssim \| u_1\|_{r, -s_1, I_1,\ast} \| u_2\|_{-s_2, I_2,\ast} \ .
	\end{equation*}
\end{lemma}
For positive Sobolev indices one has:
\begin{lemma}\label{lemma3.3a}
	Let $r\geq 0$, $0\leq s\leq 1$, $I_j = [0, h_j],\; u_2\in\widetilde{H}^{s}(I_2)$, $u_1 \in {H}^r_\sigma(\mathbb{R}^+, \widetilde{H}^{s}(I_1))$. Then there holds
	\begin{equation*}	
		\| u_1(t,x) u_2(y) \|_{r, s, I_1\times I_2, \ast} \lesssim \| u_1\|_{r, s, I_1, \ast} \| u_2\|_{s,I_2,\ast}\ .
	\end{equation*}
\end{lemma}
We also note the variants:
\begin{lemma}\label{lemma3.3v}
	Let $r\geq 0$, $0\leq s\leq 1, \; u_2\in\widetilde{H}^{-s}(\Gamma)$, $u_1 \in \widetilde{H}^r_\sigma(\mathbb{R}^+)$. Then there holds
	\begin{equation*}	
		\| u_1(t) u_2(x,y) \|_{r, -s, \Gamma, \ast} \lesssim \| u_1\|_{H^r_\sigma(\mathbb{R}^+)} \| u_2\|_{-s, \Gamma, \ast} \ .
	\end{equation*}
\end{lemma}
\begin{proof}
This is a consequence of the estimate $${(\sigma^2+|\omega|^2)^{r/2}}(\sigma^2+|\omega|^2 + \xi_1^2 + \xi_2^2)^{-s/2} \lesssim  {(\sigma^2+|\omega|^2)^{r/2}}(1+ \xi_1^2+\xi_2^2 )^{-s/2}\ $$
in Fourier space.
\end{proof}
We note a similar result for positive Sobolev indices:
\begin{lemma}\label{lemma3.3av}
	Let $r\geq 0$, $0\leq s\leq 1$, $u_2\in\widetilde{H}^{s}(\Gamma)$, $u_1 \in {H}^r_\sigma(\mathbb{R}^+)$. Then there holds
	\begin{equation*}	
		\| u_1(t) u_2(x,y) \|_{r, s, \Gamma, \ast} \lesssim \| u_1\|_{H^{r+s}_\sigma(\mathbb{R}^+)} \| u_2\|_{s,\Gamma,\ast}\ .
	\end{equation*}
\end{lemma}
\begin{proof}
This is a consequence of the estimate $${(\sigma^2+|\omega|^2)^{r/2}}(\sigma^2+|\omega|^2 + \xi_1^2 + \xi_2^2)^{s/2} \lesssim (\sigma^2+|\omega|^2  )^{{(r+s)}/2} (1+ \xi_1^2+ \xi_2^2 )^{s/2}$$ in Fourier space.
\end{proof}

\section{Discretization}
\label{Discretisationp}

For the time discretization we consider a uniform decomposition of the time interval $[0,\infty)$ into subintervals $[t_{n-1}, t_n)$ with time step $\Delta t$, such that $t_n=n\Delta t \; (n=0,1,\dots)$. \\

In $\mathbb{R}^3$, we assume that $\Gamma$ consists of } closed triangular faces $\Gamma_i$ such that $\Gamma=\cup_{i} \Gamma_i$. \textcolor{black}{Let $\mathcal{T}_S={\{\Delta_1,\cdots,\Delta_{N}\}}$ be a quasi-uniform triangulation of $\Gamma$ and $\mathcal{T}_T=\{[0,t_1),[t_1,t_2),\cdots,$ $[t_{M-1},T)\}$ the time mesh for a finite subinterval $[0,T)$.} We choose a basis $\{ \xi_h^1 , \cdots , \xi_h^{N_s}\}$ of the space $V_h^q(\Gamma)$ of piecewise polynomial functions \textcolor{black}{on $\mathcal{T}_S$} of degree $q\geq 0$ in
space \textcolor{black}{(not necessarily continuous)}. Moreover, we define $\widetilde{V}_h^q(\Gamma)$ as the subspace of $V_h^q(\Gamma)$, where the piecewise polynomials \textcolor{black}{are continuous and} vanish on $\partial \Gamma$ for $q\geq 1$. \textcolor{black}{The parameter $h$ denotes the maximal diameter of a triangle in $\mathcal{T}_S$.} For the time discretization we choose a basis $\{\beta_{\Delta t}^1,\cdots,\beta_{\Delta t}^{N_t}\}$  of  the space $V^p_{t}$ of piecewise  polynomial  functions of degree of $p$ in time (continuous and vanishing at $t=0$ if $p\geq 1$).
\\

We consider the tensor product of the approximation spaces in space and time, $V_h^q$ and $V^p_{\Delta t}$, associated to the space-time mesh $\mathcal{T}_{S,T}=\mathcal{T}_S \times\mathcal{T}_T$, and we write
\begin{align}\label{fespac}
V_{\Delta t,h}^{p,q}:=  V_{\Delta t}^p \otimes V_{h}^q\ .
\end{align}
We analogously define
\begin{align}\label{fespac2}
\tilde{V}_{\Delta t,h}^{p,q}:=  V_{\Delta t}^p \otimes \tilde{V}_{h}^q\ .
\end{align}

The Galerkin discretization of the Dirichlet problem \eqref{weakform} is then given by:\\

\noindent Find $\psi_{\Delta t, h} \in V_{\Delta t,h}^{p,q}$ such that for all $\Psi_{\Delta t, h}\in V_{\Delta t,h}^{p,q}$
\begin{equation}\label{weakformhp}
 \int_0^\infty\int_\Gamma (V \psi_{\Delta t, h}(t,{x})) \partial_t\Psi_{\Delta t, h}(t,{x}) \ dx\ d_\sigma t = \int_0^\infty \int_\Gamma f(t,{x}) \partial_t\Psi_{\Delta t, h}(t,{x}) \ dx\ d_\sigma t\ .
\end{equation}

For the Neumann problem \eqref{weakformW}, we have:\\

\noindent Find $\phi_{\Delta t, h} \in \widetilde{V}_{t,h}^{p,q}$ such that for all $\Phi_{\Delta t, h}\in \widetilde{V}_{t,h}^{p,q}$
\begin{equation}\label{weakformWhp}
 \int_0^\infty\int_\Gamma (W \phi_{\Delta t, h}(t,{x})) \partial_t\Phi_{\Delta t, h}(t,{x}) \ dx\ d_\sigma t = \int_0^\infty \int_\Gamma g(t,{x}) \partial_t\Phi_{\Delta t, h}(t,{x}) \ dx\ d_\sigma t\ .
\end{equation}

From the weak coercivity of $V$, respectively $W$, the discretized problems \eqref{weakformhp} and \eqref{weakformWhp} admit unique solutions.\\

Practical computations use $\sigma=0$. \textcolor{black}{The resulting system of equations in this case has a block-Toeplitz structure which requires to compute only one matrix per time step, unlike for $\sigma>0$. See \cite{jr} for a detailed analysis of the role of the weight $\sigma$.}

\subsection{Approximation properties}

While we use triangular meshes in our computations, for the ease of presentation we first discuss the approximation properties of meshes with rectangular elements. Reference \cite{disspetersdorff} shows how to deduce approximation results on triangular meshes from the rectangular case. 

Key ingredients in our analysis are projections from $L^2(\Gamma)$ onto $V_h^p$. We collect some  key approximation properties used below, which are proven analogous to \cite[Proposition 3.54 and 3.57]{glaefke}, see also \cite{HGEPSN} for screens.

We recall the well-known results for $V_h^p$ and $V^q_{\Delta t}$, which we are going to need. See, for example, \textcolor{black}{\cite{BamHa} for the following Lemma \ref{proj_time}. The second inequality in Lemma \ref{proj_space} may be found as Theorem 4.1 in \cite{bh08}; it implies the first one.}
\begin{lemma}\label{proj_time}
Let $\Pi_{\Delta t}^q$ the orthogonal projection from $L^2(\mathbb{R}^+)$ to $V^q_{\Delta t}$ and $m \leq q$. Then for $s \in [-\frac{1}{2}, \frac{1}{2}]$
\begin{equation*}
 ||f-\Pi_{\Delta t}^qf||_{H^r_\sigma(\mathbb{R}^+)}\textcolor{black}{\lesssim} \left(\frac{\Delta t}{q+1}\right)^{q+1-s} \|f\|_{H^{q+1}_\sigma(\mathbb{R}^+)}\,.
\end{equation*}
\textcolor{black}{for all $f \in H^{q+1}_\sigma(\mathbb{R}^+)$.}
\end{lemma}
\begin{lemma}\label{proj_space}
Let $\widetilde{\Pi}_{h,x}^p$ the orthogonal projection from $L^2(\Gamma)$ to $\widetilde{V}_h^p$ and $m\leq p$. Then for $\varepsilon>0$ and $s \in [-1,0]$ we have in the norms of $H^{s}(\Gamma)$ respectively $\widetilde{H}^{s}(\Gamma)$:
\begin{align}
 ||f-\widetilde{\Pi}_{h,x}^pf||_{s, \Gamma} &\leq C\left(\frac{h}{p+1}\right)^{m+1-s} \textcolor{black}{\|}f\textcolor{black}{\|}_{m+1, \Gamma}\ 
\end{align}
\textcolor{black}{for all $f \in H^{m+1}(\Gamma)$,}
\begin{align*}
 ||f-\widetilde{\Pi}_{h,x}^pf||_{s, \Gamma,\ast} &\leq C \left(\frac{h}{p+1}\right)^{m+1-s}\textcolor{black}{\|}f\textcolor{black}{\|}_{m+1, \Gamma}\ 
\end{align*}
for all $f \in H^{m+1}(\Gamma)\cap \widetilde{H}^{s}(\Gamma)$.
\end{lemma}

Combining $\widetilde{\Pi}_{h, x}^p$ and $\Pi_{\Delta t}^q$ one obtains as in Proposition 3.54 of \cite{glaefke}:
\begin{lemma}\label{approxprop}
Let $f \in H^{s}_\sigma(\mathbb{R}^+, H^m(\Gamma){\cap \widetilde{H}^{r}(\Gamma)})$, $0<m\leq p+1$, $0<s\leq p+1$, $r\leq s$, $|l|\leq \frac{1}{2}$ such that $lr\geq 0$. Then if  $ l,r\leq 0$ and $\varepsilon>0$
\begin{align}\label{eq:approx}
\|f-\widetilde{\Pi}_{h,x}^p  \Pi_{\Delta t}^p f\|_{r,l,\Gamma} &\leq C \left(\left(\frac{h}{p+1}\right)^\alpha + \left(\frac{\Delta t}{p+1}\right)^\beta\right)||f||_{s,m,\Gamma}\ ,\\
\|f-\widetilde{\Pi}_{h,x}^p  \Pi_{\Delta t}^p f\|_{r,l,\Gamma,\ast} &\leq C  \left(\left(\frac{h}{p+1}\right)^{\alpha-\varepsilon} + \left(\frac{\Delta t}{p+1}\right)^\beta\right)||f||_{s,m,\Gamma}\ ,
\end{align}
where $\alpha = \min\{m-l, m-\frac{m(l+r)}{m+s}\}$, $\beta = \min\{m+s-(l+r), m+s-\frac{m+s}{m}l\}$. If $l,r>0$, $\beta = m+s-(l+r)$.
\end{lemma}
\textcolor{black}{Lemma \ref{approxprop} is mostly applied} for $\Delta t \lesssim h$, when $\Delta t$ may be replaced by $h$. \textcolor{black}{The first inequalities in Lemma \ref{proj_space} and Lemma \ref{approxprop} hold verbatim also for the orthogonal projection $\Pi_{h,x}^p$  from $L^2(\Gamma)$ to $V_h^p$.}

The proof of the following result is given in  \cite[Theorem 3.1]{heuer01}  \textcolor{black}{for the $p$-version and in \cite[Theorem 3.3]{besp} for $hp$:} 
\begin{lemma}\label{keylemmagrad}For $\varepsilon > 0$, $a<1$ and $s \in [-1, \min\{-a+\frac{1}{2},0\})$ there holds with the piecewise polynomial \textcolor{black}{Lagrange} interpolant of degree $p$, $ \Pi_{h,y}^{p} y^{-a}$, of $ y^{-a}$ \textcolor{black}{on a quasi-uniform mesh of mesh size $h$}:
$$\|y^{-a} - \Pi_{h,y}^{p} y^{-a}\|_{s,[0,1],\ast} \lesssim  \left(\frac{h}{(p+1)^2}\right)^{-a+\frac{1}{2}-s-\varepsilon} . $$\ 
\end{lemma}
For positive powers of $y$ we use \cite[Theorem 3.1]{heuer2} \textcolor{black}{for the $p$-version and in \cite[Theorem 3.2]{besp} for $hp$:}
\begin{lemma}\label{keylemmagrad2}For $\varepsilon > 0$, $0<a$ and $s \in [0, a+\frac{1}{2})$ there holds with the piecewise polynomial \textcolor{black}{Lagrange} interpolant of degree $p+1$, $ \widetilde{\Pi}_{h,y}^{p+1} y^{a}$, of $ y^{a}$ \textcolor{black}{on a quasi-uniform mesh of mesh size $h$}:
$$\|y^{a} - \widetilde{\Pi}_{h,y}^{p+1} y^a\|_{s,[0,1],\ast} \lesssim \left(\frac{h}{p^2}\right)^{\min\{a+\frac{1}{2}-s, 2-s\}{-\varepsilon}} . $$
\end{lemma}

\section{Approximation of singularities}
\label{sectasympt}

Solutions of the wave equation \eqref{eq:oriProblemp} exhibit  singularities at edges and corners of the domain. We here recall a decomposition of the solution near these non-smooth boundary points  into a leading part given by explicit singular functions plus less singular terms.  

Let $0\leq d\leq n-2$ and $K \subset \mathbb{R}^{n-d}$ an open cone with vertex at $0$, which is smooth outside the vertex. Denote the wedge over $K$ by $\mathcal{K} = K \times \mathbb{R}^d$. We study the wave equation in $\mathcal{K}$: 
\begin{subequations} \label{eq:oriProblem2}
\begin{alignat}{2}
\partial_t^2 u(t,x) -\Delta_x u(t,x)&=0 &\quad &\text{in }\mathbb{R}^+_t \times \mathcal{K}_x\ , \\
{u}(0,x)=\partial_t u(0,x)&=0 &\quad &\text{in } \mathcal{K} , 
\end{alignat}
\end{subequations}
with either inhomogeneous Dirichlet boundary conditions \textcolor{black}{$u|_\Gamma = g$} or Neumann boundary conditions \textcolor{black}{$\partial_n u|_\Gamma = g$ on $\Gamma = \partial \mathcal{K}$.}  We aim to describe the asymptotic behavior of a solution in $\mathcal{K}$ near $\{0\}\times \mathbb{R}^d$. Locally, the edge of a screen in $\mathbb{R}^3$ corresponds to $d=1$, a cone point to $d=0$.

After a separation of variables near the edge of $\mathcal{K}$, we consider the operator $\mathfrak{A}_{B}(\nu) = \nu^2+(n-d-2)\nu-\Delta_{S}$ with $B=D$ for Dirichlet and $B=N$ for Neumann boundary conditions in the subset $\Xi=K \cap S^{n-d-1}$ of the sphere. $\Delta_S$ is the  Laplace operator on $S^{n-d-1}$, and its  eigenvalues in $\Xi$ are denoted by $\{\mu_{k,B}\}_{k=0}^\infty$. The eigenvalues of $\mathfrak{A}_B(\nu)$ may then be expressed as $-i\nu_{\pm k,B} = \frac{i(n-d-2)}{2}\mp i \lambda_{k,B}$ with $\lambda_{k,B}= \frac{((n-d-2)^2+4\mu_{k,B})^{1/2}}{2}$. We normalize the associated orthogonal eigenfunctions $\Phi_{k,B}$ of the angular variables $\theta$ as $\|\Phi_{k,B}\|^2_{L^2(\Xi)} = \lambda_{k,B}^{-1}$.

For $d=1$, $n=3$, the nonzero eigenvalues $-i\nu_{\pm k,B}= \mp \frac{k\pi}{\alpha}$ are simple provided $\frac{k\pi}{\alpha} \not \in \mathbb{N}$, where $\alpha$ denotes the opening angle of $K\subset\mathbb{R}^2$. They have multiplicity $2$ otherwise. For $k>0$ one has the explicit formulas $\Phi_{k,N}(\theta) = (k\pi)^{-\frac{1}{2}} \cos(k\pi \theta/\alpha)$, $\Phi_{k,D}(\theta) = (k\pi)^{-\frac{1}{2}} \sin(k\pi \theta/\alpha)$. In the case of Neumann boundary conditions, the eigenvalue $-i\nu_{0,N}=0$ has multiplicity $2$. 

The limit $\alpha$ tends to $2\pi^{-}$ recovers a screen with flat boundary, and for circular edges one may adapt the discussion as in \cite{petersdorff3}. 

For $d=0$, $n=3$, the singular exponents in the corner need to be determined numerically. See \cite{stephanwhite} for a discussion of polyhedral domains.

The singular exponents determine the local asymptotic expansion of the solution to the inhomogeneous wave equation 
\begin{subequations} \label{eq:oriProblemInhom}
\begin{alignat}{2}
\partial_t^2 u(t,x) -\Delta_x u(t,x)&=f &\quad &\text{in }\mathbb{R}^+_t \times \mathcal{K}_x\ , \\
{u}(0,x)=\partial_t u(0,x)&=0 &\quad &\text{in } \mathcal{K} , 
\end{alignat}
\end{subequations}
near the singular points. \textcolor{black}{Here $u|_\Gamma=0$ on $\Gamma = \partial \mathcal{K}$ for Dirichlet, $\partial_n u|_\Gamma = 0$ for Neumann boundary conditions}. For details, see \cite[Theorem 7.4 and Remark 7.5]{kokotov} in the case of the Neumann problem in a wedge, and \cite[Theorem 4.1]{kokotov3} for the Dirichlet problem in a cone. The formulas for the asymptotic expansion   involve special solutions of the Dirichlet or Neumann problem in $K$, as in \cite[(3.5)]{kokotov3}, respectively \cite[(4.4)]{kokotov}: 
$$w_{-k,B}(y, \omega, \zeta) = \frac{2^{1-\lambda_{k,B}}}{\Gamma(\lambda_{k,B})}(i|y|\sqrt{-|\zeta|^2+\omega^2})^{\lambda_{k,B}} K_{\lambda_{k,B}}(i|y|\sqrt{-|\zeta|^2+\omega^2}) |y|^{\nu_{{-k,B}}} \Phi_{k,B}(y/|y|) \ .$$
Here $K_\lambda$ denotes the modified Bessel function of the third kind.

Using $w_{-k,B}$, the leading singularities near an edge or a cone point are given as 
\begin{align}\label{singexp} |y|^{ \nu_{j,B}} \Phi_{j,B}(\theta) (\partial_t^2-\Delta_z)^m(i|y|)^{2m} {\mathcal{F}^{-1}_{(\omega, \zeta) \to (t,z)}{c}_{j,B}} \ ,\end{align}
with $c_{j,B}(\omega, \zeta) = \langle \hat{f}(\cdot, \omega, \zeta), w_{-j,B}(\cdot, \overline{\omega}, \zeta)\rangle_{L^2(K)}$, plus a remainder which is less singular \cite{kokotov, kokotov3}. Here $m\in \mathbb{N}$. Additional logarithmic terms in $|y|$ appear if  $i \nu_{j,B} \in \mathbb{N}$. The regularity of $c_{j,B}(\omega, \zeta) = \langle \hat{f}(\cdot, \omega, \zeta), w_{-j,B}(\cdot, \overline{\omega}, \zeta)\rangle_{L^2(K)}$  is determined by the data $f$ in the wave equation \eqref{eq:oriProblemInhom}. \\

\textcolor{black}{A precise statement of the asymptotic expansion is the content of Theorem \ref{app:theorem2} in the Appendix. the Appendix also provides a discussion of the proof. Using this theorem one sees that by expanding sufficiently many terms, the remainder in the expansion can be made to have  order  $|y|^\eta$, for any $\eta>0$, with a smooth coefficient if the right hand side of the equation is smooth. Knowing the precise Sobolev regularity of the coefficient of this singular function, depending on the Sobolev regularity of the right hand side, would be of interest: It would allow to state precise smoothness assumptions on $g$ in Theorem \ref{approxtheorem1}, resp. Theorem \ref{approxtheorem2}). However, this refined analysis is beyond the scope of the current article.}\\

Further information can be obtained by combining the convolution representation
$${\mathcal{F}^{-1}_{(\omega, \zeta) \to (t,z)}{c}_{j,B}} = \int_{\mathbb{R}^d}\int_{\mathbb{R}} \int_K  f(y,z_1,t_1) W_{-j,B}(y, {t-t_1, z-z_1})\ dy \ dt_1 \ dz_1$$ 
with information about the singular functions $W_{-{j},B}(y,t,z) = \mathcal{F}^{-1}_{(\zeta,\omega) \to (t,z)}w_{-{j},B}$. 
The singular support of $W_{-j,B}$ lies on a light cone emanating from the edge, $\{(y,t,z) \in \mathbb{R}^{n+1} : t = \sqrt{|y|^2+|z|^2}\}$. Therefore {$\mathcal{F}^{-1}_{(\omega, \zeta) \to (t,z)}{c}_{j,B}$} is smooth in $$\{(t,z)\in\mathbb{R}^{d+1} : t>\sup\{t_1 + \sqrt{|y|^2+|z-z_1|^2}  : (y,z_1,t_1) \in \mathrm{singsupp}\ f\}\}\ .$$ For smooth $f$, $\mathrm{singsupp}\ f = \emptyset$ and therefore also {$\mathcal{F}^{-1}_{(\omega, \zeta) \to (t,z)}{c}_{j,B}$} is smooth everywhere.\\

The \textcolor{black}{singular} expansion for the inhomogeneous wave equation in \eqref{eq:oriProblemInhom} \textcolor{black}{in $\mathbb{R}^+_t \times \mathcal{K}_x$} implies an expansion for inhomogeneous  boundary conditions in \eqref{eq:oriProblem2}. The argument is as for elliptic problems \cite[Section 5]{petersdorff2}: For \textcolor{black}{Dirichlet conditions $u|_\Gamma = g$} on ${\mathbb{R}^+_t \times} \partial \mathcal{K}$, we choose an extension $\widetilde{g}$ in ${\mathbb{R}^+_t \times}\mathcal{K}$ with $\widetilde{g}|_\Gamma = g$ on ${\mathbb{R}^+_t \times}\partial \mathcal{K}$. Then $U = u-\widetilde{g}$ satisfies the inhomogeneous wave equation $\partial_t^2 U - \Delta_x U = {f} -\partial_t^2 \widetilde{g} + \Delta \widetilde{g}$ with homogeneous boundary condition $\textcolor{black}{U|_\Gamma}=0$. The above discussion describes the asymptotic expansion of $U$, and one concludes a corresponding expansion for $u = U + \widetilde{g}$. \textcolor{black}{An analogous argument applies to Neumann boundary conditions $\partial_n u|_\Gamma = g$.} 

The resulting asymptotic expansions of the boundary values $u|_\Gamma$ and $\partial_n u|_\Gamma$ will be crucial for the analysis of the solutions to the boundary integral formulations, \textcolor{black}{and for the ease of the reader we give more details in the Appendix.}

In the case of a wedge, regularity results have also been obtained by  Eskin \cite{eskin} using Wiener-Hopf symbol factorizations. 

\subsection{Singularities for circular screens and approximation}

We first illustrate the above expansion for the exterior of a circular wedge with exterior opening angle $\alpha$. For $\alpha \to 2\pi^-$, the wedge degenerates into the circular screen $\{(x_1,x_2,0) \in \mathbb{R}^3 :  {x_1^2+x_2^2}\leq 1\}$. Near the edge $\{(x_1,x_2,0) \in \mathbb{R}^3 :  {x_1^2+x_2^2}= 1\}$ we use the coordinates $(y,z,\theta)$, where in polar coordinates in the $x_1-x_2$-plane $y=r-1$, $z=\theta$. Using \cite{petersdorff3}, an analogous expansion to  \eqref{eq:oriProblem2} also holds in this curved geometry, with the same leading singular term $|y|^{\nu}$, where $\nu\to \frac{1}{2}$ as $\alpha \to 2\pi^-$:
\begin{align}
u(y,t,z)|_\Gamma &= a(t,z) |y|^{\frac{1}{2}} + v_0(y,z,t)\  , \label{decompostionEdget}\\
\partial_n u(y,t,z)|_\Gamma &= b(t,z) |y|^{-\frac{1}{2}} + {\psi_0}(y, z,t)\ . \label{decompositionEdge}
 \end{align}
Here $a$ and $b$ are smooth for smooth data.
 
From these decompositions we obtain \textcolor{black}{quasi-}optimal approximation properties for the $hp$-version, \textcolor{black}{up to an arbitrarily small $\varepsilon>0$. }\\

\begin{theorem}\label{approxtheorem1} {Let $\varepsilon>0$.}
a) Let $u$ be a  solution to the homogeneous wave equation with inhomogeneous Neumann boundary conditions $\partial_n u|_\Gamma = g$, \textcolor{black}{with $g \in {H}^{\alpha}_\sigma(\R^+, \widetilde{H}^{\beta}(\Gamma))$ for some $\alpha,\beta$, so that the regular part $v_0$ \textcolor{black}{belongs to} ${H}^{\mu}_\sigma(\R^+, \widetilde{H}^{\eta}(\Gamma))$  in the singular expansion of $u|_\Gamma$, with $\eta, \mu$ sufficiently large}. Further, let $\phi_{h,\Delta t}$ be the best approximation {in the norm of ${H}^{r}_\sigma(\R^+, \widetilde{H}^{\frac{1}{2}-s}(\Gamma))$ to the Dirichlet trace $u|_\Gamma$ in $\widetilde{V}^{p,p}_{\Delta t,h}$} on a quasi-uniform spatial mesh with $\Delta t \lesssim h$. Then   $$\|u-\phi_{h, \Delta t}\|_{r,\frac{1}{2}-s, \Gamma, \ast} \lesssim \left(\frac{h}{p^2}\right)^{\frac{1}{2}+s{-\varepsilon}} + \left(\frac{h}{p}\right)^{-\frac{1}{2}+s+\eta}+\left(\frac{\Delta t}{p}\right)^{\mu+s-r-\frac{1}{2}}\ ,$$
 where {$r \in [0,p)$}.

b) Let $u$ be a  solution to the homogeneous wave equation with inhomogeneous Dirichlet boundary conditions $u|_\Gamma = g$, \textcolor{black}{with $g \in {H}^{\alpha}_\sigma(\R^+, \widetilde{H}^{\beta}(\Gamma))$ for some $\alpha,\beta$, so that the regular part $\psi_0$ \textcolor{black}{belongs to} ${H}^{\mu}_\sigma(\R^+, \widetilde{H}^{\eta}(\Gamma))$  in the singular expansion of $\partial_n u|_\Gamma$, with $\eta, \mu$ sufficiently large}. Further, let $\psi_{h,\Delta t}$ be the best approximation  {in the norm of ${H}^{r}_\sigma(\R^+, \widetilde{H}^{-\frac{1}{2}}(\Gamma))$ to the Neumann trace $\partial_n u|_\Gamma$ in ${V}^{p,p}_{\Delta t,h}$} on a quasi-uniform spatial mesh  with $\Delta t \lesssim h$. Then  $$\|\partial_n u - \psi_{h,\Delta t}\|_{r, -\frac{1}{2}, \Gamma, \ast} \lesssim \left(\frac{h}{(p+1)^2}\right)^{\frac{1}{2}{-\varepsilon}} + \left(\frac{h}{p+1}\right)^{\frac{1}{2}+\eta}+ \left(\frac{\Delta t}{p+1}\right)^{\mu+1-r}\ ,$$
 where {$r \in [0,p+1)$}.
\end{theorem}

Theorem \ref{approxtheorem1} implies a corresponding result for the solutions of the single layer and hypersingular integral equations on the screen:
\begin{corollary}\label{approxcor1} Let $\varepsilon>0$.
a) Let $\phi$ be the solution to the hypersingular integral equation \eqref{hypersingeq} and  $\phi_{h,\Delta t}$ the best approximation {in the norm of ${H}^{r}_\sigma(\R^+, \widetilde{H}^{\frac{1}{2}-s}(\Gamma))$ to $\phi$ in $\widetilde{V}^{p,p}_{\Delta t,h}$} on a quasi-uniform spatial mesh  with $\Delta t \lesssim h$. \textcolor{black}{Assume that the right hand side $g \in {H}^{\alpha}_\sigma(\R^+, \widetilde{H}^{\beta}(\Gamma))$ for some $\alpha,\beta$, so that the regular part $v_0$ \textcolor{black}{belongs to} ${H}^{\mu}_\sigma(\R^+, \widetilde{H}^{\eta}(\Gamma))$  in the singular expansion of $u|_\Gamma$, with $\eta, \mu$ sufficiently large}. Then  $$\|\phi-\phi_{h, \Delta t}\|_{r,\frac{1}{2}-s, \Gamma, \ast} \lesssim \left(\frac{h}{p^2}\right)^{\frac{1}{2}+s{-\varepsilon}} + \left(\frac{h}{p}\right)^{-\frac{1}{2}+s+\eta}+\left(\frac{\Delta t}{p}\right)^{\mu+s-r-\frac{1}{2}}\ ,$$
 where $r \in [0,p)$, $s\in [0,\frac{1}{2}]$.

b) Let $\psi$ be the solution to the single layer integral equation \eqref{dirproblemV} and  $\psi_{h,\Delta t}$ the best approximation  {in the norm of ${H}^{r}_\sigma(\R^+, \widetilde{H}^{-\frac{1}{2}}(\Gamma))$ to $\psi$ in ${V}^{p,p}_{\Delta t,h}$} on a quasi-uniform spatial mesh  with $\Delta t \lesssim h$. \textcolor{black}{Assume that the right hand side $f \in {H}^{\alpha}_\sigma(\R^+, \widetilde{H}^{\beta}(\Gamma))$ for some $\alpha,\beta$, so that the regular part $\psi_0$ \textcolor{black}{belongs to} ${H}^{\mu}_\sigma(\R^+, \widetilde{H}^{\eta}(\Gamma))$  in the singular expansion of $\partial_n u|_\Gamma$, with $\eta, \mu$ sufficiently large}. Then $$\|\psi - \psi_{h,\Delta t}\|_{r, -\frac{1}{2}, \Gamma, \ast} \lesssim \left(\frac{h}{(p+1)^2}\right)^{\frac{1}{2}{-\varepsilon}} + \left(\frac{h}{p+1}\right)^{\frac{1}{2}+\eta}+ \left(\frac{\Delta t}{p+1}\right)^{\mu+1-r}\ ,$$
 where {$r \in [0,p+1)$}.
\end{corollary}
Indeed, on the flat screen the solutions to the integral equations are given by $\phi = \left[u\right]|_\Gamma$ {in terms of the solution $u$ which satisfies Neumann conditions $Bu = \partial_n u|_\Gamma = g$, respectively} $\psi = \left[\partial_n u\right]|_\Gamma$ {in terms of the solution $u$ which satisfies Dirichlet conditions $Bu = u|_\Gamma=f$}.\\

\textcolor{black}{The  proof of Theorem \ref{approxtheorem1}  is the content of}  the following two subsections. 

\subsubsection{Approximation of the Neumann trace}\label{secneum}

\begin{theorem} 
Under the assumptions of Theorem \ref{approxtheorem1}, for $\Delta t \lesssim h$ there holds \textcolor{black}{for $r \in [0,p+1)$}
$$\|\partial_n u - \Pi_{h,x}^p \Pi_{\Delta t}^p \partial_n u\|_{r, -\frac{1}{2}, \Gamma, \ast} \lesssim \left(\frac{h}{(p+1)^2}\right)^{\frac{1}{2}{-\varepsilon}} + \left(\frac{h}{p+1}\right)^{\frac{1}{2}+\eta}+  \left(\frac{\Delta t}{p+1}\right)^{\mu+1-r}\ .$$
\end{theorem}
\begin{proof}{Using the decomposition \eqref{decompositionEdge} for $ \partial_n u$, we can separate the singular and regular parts on the rectangular mesh:
\begin{align*}&\|\partial_n u - \Pi_{h,x}^{p} \Pi_{\Delta t}^{{p}} \partial_n u\|_{r, -\frac{1}{2}, \Gamma, \ast} \leq \|b(t,z) |y|^{-\frac{1}{2}} - \Pi_{\Delta t}^{{p}} \Pi_{h,x}^{p} b(t,z) |y|^{-\frac{1}{2}}\|_{r, -\frac{1}{2}, \Gamma,\ast}  + \|\psi_0 - \Pi_{\Delta t}^{{p}} \Pi_{h,x}^{p}{\psi_0}\|_{r, -\frac{1}{2}, \Gamma,\ast}\\ & \leq \|b(t,z) |y|^{-\frac{1}{2}} -\Pi_{\Delta t}^{{p}} b(t,z) |y|^{-\frac{1}{2}}\|_{r, -\frac{1}{2}, \Gamma,\ast}+\|\Pi_{\Delta t}^{{p}} b(t,z) |y|^{-\frac{1}{2}}- \Pi_{\Delta t}^{{p}} \Pi_{h,x}^{p} b(t,z) |y|^{-\frac{1}{2}}\|_{r, -\frac{1}{2}, \Gamma,\ast} \\ & \qquad + \|\psi_0 - \Pi_{\Delta t}^{{p}} \Pi_{h,x}^{p}{\psi_0}\|_{r, -\frac{1}{2}, \Gamma,\ast}\\ &
\leq \|b(t,z) -\Pi_{\Delta t}^{{p}} b(t,z)\|_{r, \epsilon-\frac{1}{2}} \||y|^{-\frac{1}{2}}\|_{-\varepsilon, I, \ast}  +  \|\Pi_{\Delta t}^{{p}} b(t,z) |y|^{-\frac{1}{2}}- \Pi_{\Delta t}^{{p}} \Pi_{h,z}^{p} b(t,z) |y|^{-\frac{1}{2}}\|_{r, -\frac{1}{2}, \Gamma,\ast} \\ & \qquad + \| \Pi_{\Delta t}^{{p}} \Pi_{h,z}^{p} b(t,z) |y|^{-\frac{1}{2}}-\Pi_{\Delta t}^{{p}} \Pi_{h,z}^{p} b(t,z) \Pi_{h,y}^{p} |y|^{-\frac{1}{2}}\|_{r, -\frac{1}{2}, \Gamma,\ast}+
\|\psi_0 - \Pi_{\Delta t}^{{p}} \Pi_{h,x}^{p}{\psi_0}\|_{r, -\frac{1}{2}, \Gamma,\ast}\ .
\end{align*}
Here, for the first term we have used Lemma \ref{lemma3.3}, and for the second $\Pi_{h,x}^{p} =  \Pi_{h,z}^{p} \Pi_{h,y}^{p}$. The norm $\|\cdot\|_{r,\epsilon-\frac{1}{2},I}$ is the anisotropic space-time Sobolev norm in the $t$ and $z$ coordinates. We note that the first term is bounded by $$\|b(t,z) -\Pi_{\Delta t}^{{p}} b(t,z)\|_{r, \epsilon-\frac{1}{2}} \lesssim \left(\frac{ \Delta t}{p+1}\right)^{\mu+1-r}   \|b(t,z)\|_{\mu+1, \epsilon-\frac{1}{2}}\ .$$
For the second and third terms we obtain with Lemma \ref{lemma3.3}:
\begin{align*}&
 \|\Pi_{\Delta t}^{{p}} b(t,z) |y|^{-\frac{1}{2}}- \Pi_{\Delta t}^{{p}} \Pi_{h,z}^{p} b(t,z) |y|^{-\frac{1}{2}}\|_{r, -\frac{1}{2}, \Gamma,\ast} + \| \Pi_{\Delta t}^{{p}} \Pi_{h,z}^{p} b(t,z) |y|^{-\frac{1}{2}}-\Pi_{\Delta t}^{{p}} \Pi_{h,z}^{p} b(t,z) \Pi_{h,y}^{p} |y|^{-\frac{1}{2}}\|_{r, -\frac{1}{2}, \Gamma,\ast}\\ & \lesssim \|\Pi_{\Delta t}^{{p}} b(t,z) - \Pi_{\Delta t}^{{p}} \Pi_{h,z}^{p} b(t,z)\|_{r, \varepsilon - \frac{1}{2} } \||y|^{-\frac{1}{2}}\|_{-\varepsilon, I, \ast}  +\| \Pi_{\Delta t}^{{p}} \Pi_{h,z}^{p} b(t,z)\|_{r,0} \||y|^{-\frac{1}{2}} - \Pi_{h,y}^{p}|y|^{-\frac{1}{2}}\|_{-\frac{1}{2}, I, \ast } \ .
\end{align*}}
{From Lemma \ref{keylemmagrad} we have $\|y^{-a} - \Pi_{y}^{p} y^{-a}\|_{-\frac{1}{2}, I, \ast} \lesssim  \left(\frac{h}{(p+1)^2}\right)^{-a+1-\varepsilon}$ and $$\|\Pi_{\Delta t}^{{p}} b(t,z) - \Pi_{\Delta t}^{{p}} \Pi_{h,z}^{p} b(t,z)\|_{r, \varepsilon - \frac{1}{2} } \lesssim \left(\frac{h}{p+1}\right)^{\frac{1}{2}+k- \epsilon}\|b(t,z)\|_{r,k} \ . $$ }

After possibly expanding finitely many terms, which may be treated as above, we assume that the regular part $\psi_0$ in \eqref{decompositionEdge} is $H^\eta$ in space. Then
using the approximation properties for $\psi_0$, 
\begin{align*}\|\psi_0 - \Pi_{\Delta t}^{{p}} \Pi_{h,x}^{p}{\psi_0}\|_{r, -\frac{1}{2}, \Gamma,\ast} &\lesssim_\sigma   \Big(\left(\frac{\Delta t}{p+1}\right)^{\mu+1-r}+ \left(\frac{h}{p+1}\right)^{\frac{1}{2}+\eta}\Big)\|\psi_0\|_{\mu+1, \eta, \Gamma}\ .
\end{align*}
Combining the estimates for the different terms, we conclude that for $\Delta t \lesssim h$ and sufficiently large $k$\\
$$\|\partial_n u - \Pi_{h,x} \Pi_{\Delta t} \partial_n u\|_{r, -\frac{1}{2}, \Gamma, \ast} \lesssim \left(\frac{h}{(p+1)^2}\right)^{\frac{1}{2}{-\varepsilon}} + \left(\frac{h}{p+1}\right)^{\frac{1}{2}+\eta}+  \left(\frac{\Delta t}{p+1}\right)^{\mu+1-r}\ .$$
\end{proof}

\subsubsection{Approximation of the Dirichlet trace}

We now consider the approximation of the solution $u$ to the wave equation on the screen, with expansion \eqref{decompostionEdget}, or equivalently the solution to the hypersingular integral equation. Apart from the energy norm, here the $L^2$-norm is of interest, and we state the result for general Sobolev indices:  
\begin{theorem}
\textcolor{black}{Under the assumptions of Theorem \ref{approxtheorem1},} for $\Delta t \lesssim h$, $r \in [0,p)$ and $s\in [0, \frac{1}{2}]$ there holds
$$\|u - \widetilde{\Pi}_{h,x}^{{p}} \Pi_{\Delta t}^{{p}} u\|_{r,\frac{1}{2}-s, \Gamma, \ast} \lesssim  \left(\frac{h}{p^2}\right)^{\frac{1}{2}+s{-\varepsilon}}+\left(\frac{h}{p}\right)^{s-\frac{1}{2}+\eta} + \left(\frac{\Delta t}{p}\right)^{\mu+s-r-\frac{1}{2}}\ .$$
\end{theorem}
\begin{proof}
Following the approach in Section \ref{secneum}, we use the triangle inequality
\begin{align*}
\| u - \widetilde{\Pi}_{h,x}^{{p}} \Pi_{\Delta t}^{{p}}  u\|_{r, \frac{1}{2}-s,  \Gamma,\ast} &\leq\|a(t,z)|y|^{\frac{1}{2}} - \Pi_{\Delta t}^{{p}}a(t,z)|y|^{\frac{1}{2}}\|_{r, \frac{1}{2}-s,  \Gamma,\ast}\\ & \qquad + \|\Pi_{\Delta t}^{{p}}a(t,z)|y|^{\frac{1}{2}} - \widetilde{\Pi}_{h,x}^{{p}} \Pi_{\Delta t}^{{p}}a(t,z)|y|^{\frac{1}{2}}\|_{r, \frac{1}{2}-s,  \Gamma,\ast}+\|v_0 - \widetilde{\Pi}_{h,x}^{{p}} \Pi_{\Delta t}^{{p}}  v_0\|_{r, \frac{1}{2}-s,  \Gamma,\ast}\ .
\end{align*}
We first estimate
\begin{align*}
\|a(t,z)|y|^{\frac{1}{2}} -  \Pi_{\Delta t}^{{p}}a(t,z)|y|^{\frac{1}{2}}\|_{r, \frac{1}{2}-s,  \Gamma,\ast}&\leq \| a(t,z) - \Pi_{\Delta t}^{{p}} a(t,z)\|_{r,  \frac{1}{2}-s, \ast} \||y|^{\frac{1}{2}}\|_{\frac{1}{2}-s,\ast}
\end{align*}
and note that
$$\|a(t,z) - \Pi_{\Delta t}^{{p}}  a(t,z)\|_{r, \frac{1}{2}-s, \ast} \lesssim \left(\frac{\Delta t}{p}\right)^{\mu+s-r-\frac{1}{2}}\|a(t,z)\|_{\mu,\frac{1}{2}-s}\ .$$
For the second term we note with Lemma \ref{lemma3.3a}, respectively Lemma \ref{lemma3.3}:
\begin{align*}
&\| \Pi_{\Delta t}^{{p}} a(t,z) |y|^{\frac{1}{2}} - \Pi_{\Delta t}^{{p}} \widetilde{\Pi}_{h,x}^{{p}} a(t,z) |y|^{\frac{1}{2}}\|_{r, \frac{1}{2}-s,  \Gamma,\ast}\\
&\leq \| \Pi_{\Delta t}^{{p}} a(t,z) |y|^{\frac{1}{2}}- \Pi_{\Delta t}^{{p}} \widetilde{\Pi}_{h,z}^{{p}} a(t,z) |y|^{\frac{1}{2}} + \Pi_{\Delta t}^{{p}} \widetilde{\Pi}_{h,z}^{{p}} a(t,z) |y|^{\frac{1}{2}} - \Pi_{\Delta t}^{{p}} \widetilde{\Pi}_{h,z}^{{p}} a(t,z) \widetilde{\Pi}_{h,y}^{{p}} |y|^{\frac{1}{2}}\|_{r, \frac{1}{2}-s,  \Gamma,\ast}\\
&\leq \|\Pi_{\Delta t}^{{p}} a(t,z) - \Pi_{\Delta t}^{{p}} \widetilde{\Pi}_{h,z}^{{p}} a(t,z)\|_{r,  \frac{1}{2}-s, \ast} \||y|^{\frac{1}{2}}\|_{ \frac{1}{2}-s,I,\ast  }  +\| \Pi_{\Delta t}^{{p}} \widetilde{\Pi}_{h,z}^{{p}} a(t,z)\|_{r, \frac{1}{2}-s, \ast} \||y|^{\frac{1}{2}} - \widetilde{\Pi}_{h,y}^{{p}}|y|^{\frac{1}{2}}\|_{\frac{1}{2}-s,I, \ast} \ .
\end{align*}
Now note that
$$\|\Pi_{\Delta t}^{{p}} a(t,z) - \Pi_{\Delta t}^{{p}} \widetilde{\Pi}_{h,z}^{{p}} a(t,z)\|_{r, \frac{1}{2}-s, \ast} \lesssim \left(\frac{h}{p}\right)^{k-\frac{1}{2}+s}\|a(t,z)\|_{r,k}$$
and, from Lemma \ref{keylemmagrad2}, 
$$\||y|^{\frac{1}{2}} - \widetilde{\Pi}_{h,y}^{{p}}|y|^{\frac{1}{2}}\|_{\frac{1}{2}-s, I,\ast } \lesssim \left(\frac{h}{p^2}\right)^{\frac{1}{2}+s-\varepsilon}\ .$$
It remains to estimate the remainder $$\|v_0-\widetilde{\Pi}_{h,x}^{{p}}\Pi_{\Delta t}^{{p}}  v_0\|_{r, \frac{1}{2}-s,  \Gamma,\ast} \lesssim \left(\frac{h}{p}\right)^{s-\frac{1}{2}+\eta} + \left(\frac{\Delta t}{p}\right)^{\mu+s-r-\frac{1}{2}} \ ,$$
as in Section \ref{secneum}.
Combining the estimates for the different terms, we conclude the assertion.
\end{proof}

\subsection{Singularities for polygonal screens and approximation}

We consider the singular expansion of the solution to the wave equation \eqref{eq:oriProblemp} with Dirichlet or Neumann boundary conditions on a polygonal screen $\Gamma$. Compared to \eqref{decompostionEdget}, \eqref{decompositionEdge} additional singularities now arise from the corners of the screen. For simplicity, we restrict ourselves to the model case of a flat polygonal screen $\Gamma \subset \R^3$. In this geometry, for elliptic problems asymptotic expansions and their implications for the numerical approximation are discussed in \cite{screenMaischak, petersdorff}.

The following  gives  a decomposition of  the solution  and its normal derivative on $\Gamma$ near the vertex $(0,0)$, in terms of polar coordinates $(r,\theta)$ centered at this point \cite{graded}. Note that we have two boundary values, $u_\pm$, from the upper and lower sides of the screen, and that we use refined information about the edge-vertex singularity.

\begin{align}\label{decompositiont}\nonumber
{u(t,x)|_+} &= C(t)\chi(r)r^{\lambda} \Phi(\theta) +C_1(t)\tilde{\chi}(\theta)\beta_{1}(r)(\sin(\theta))^{\frac{1}{2}}\\& \qquad + C_2(t)\tilde{\chi}(\textstyle{\frac{\pi}{2}}-\theta)\beta_{2}(r)(\cos(\theta))^{\frac{1}{2}} +v_{0}(t,r, \theta) \\ &=: u^v + u^{ev}_1 + u^{ev}_2 + v_0\ ,\nonumber\\
\nonumber
{\partial_n u(t,x)|_+} &= C'(t)\chi(r)r^{\lambda-1} \Phi'(\theta) +C_1'(t)\tilde{\chi}(\theta)\beta_{1}'(r)r^{-1}(\sin(\theta))^{-\frac{1}{2}}\\& \qquad + C_2'(t)\tilde{\chi}(\textstyle{\frac{\pi}{2}}-\theta)\beta_{2}'(r)r^{-1}(\cos(\theta))^{-\frac{1}{2}}+ \psi_{0}(t,r, \theta) \label{decomposition}\\& =: \psi^v + \psi^{ev}_1 + \psi^{ev}_2 + \psi_0\ . \nonumber
\end{align}
Here $\beta_j(r)$ behaves like $r^{\lambda-\frac{1}{2}}$ near $r=0$, while $\beta_j'(r)$ behaves like $r^{\lambda}$, $j=1,2$, . Compared to the local coordinates near the edge in the previous section, the polar angle $\theta$ corresponds to the distance $|y|$ to the edge and the radius $r$ to the variable $z$ along the edge. For $\Gamma=(0,1)\times(0,1)\times\{0\}$, the corner exponent $\lambda \simeq 0.2966$.\\

To control the remainder terms in these formal computations requires elliptic a priori weighted estimates near the singularities, as discussed in \cite{matyu}. \\

\textcolor{black}{In the literature decompositions like \eqref{decompositiont} and \eqref{decomposition} in polar coordinates are also expressed in Cartesian coordinates or a mixture of Cartesian and polar coordinates. The Appendix of \cite{bh08} shows the equivalence of these descriptions, also remarked in Corollary 4 of \cite{petersdorff}. Below we may therefore appeal to results stated for expansions in other coordinate systems.}\\

From the decomposition, similar to Theorem \ref{approxtheorem1} we obtain the approximation properties of the $hp$-method. The error is dominated by the edge singularities, not the corners. 

\begin{theorem}\label{approxtheorem2} {Let $\varepsilon>0$.}
a) Let $u$ be a  solution to the homogeneous wave equation with inhomogeneous Neumann boundary conditions $\partial_n u|_\Gamma = g$, \textcolor{black}{with $g \in {H}^{\alpha}_\sigma(\R^+, \widetilde{H}^{\beta}(\Gamma))$ for some $\alpha,\beta$, so that the regular part $v_0$ \textcolor{black}{belongs to} ${H}^{\mu}_\sigma(\R^+, \widetilde{H}^{\eta}(\Gamma))$  in the singular expansion of $u|_\Gamma$, with $\eta, \mu$ sufficiently large}. Further, let $\phi_{h,\Delta t}$ be the best approximation {in the norm of ${H}^{r}_\sigma(\R^+, \widetilde{H}^{\frac{1}{2}-s}(\Gamma))$ to the Dirichlet trace $u|_\Gamma$ in $\widetilde{V}^{p,p}_{\Delta t,h}$} on a quasi-uniform spatial mesh with $\Delta t \lesssim h$. Then $$\|u-\phi_{h, \Delta t}\|_{r,\frac{1}{2}-s, \Gamma, \ast} \lesssim \left(\frac{h}{p^2}\right)^{\frac{1}{2}+ \min\{\lambda,0\}+s{-\varepsilon}} + \left(\frac{h}{p}\right)^{-\frac{1}{2}+s+\eta}+  \left(\frac{\Delta t}{p}\right)^{\mu+s-r-\frac{1}{2}}\ ,$$
 where {$r \in [0,p)$}.

b) Let $u$ be a  solution to the homogeneous wave equation with inhomogeneous Dirichlet boundary conditions $u|_\Gamma = g$,  \textcolor{black}{with $g \in {H}^{\alpha}_\sigma(\R^+, \widetilde{H}^{\beta}(\Gamma))$ for some $\alpha,\beta$, so that the regular part $\psi_0$ \textcolor{black}{belongs to} ${H}^{\mu}_\sigma(\R^+, \widetilde{H}^{\eta}(\Gamma))$  in the singular expansion of $\partial_n u|_\Gamma$, with $\eta, \mu$ sufficiently large}. Further, let $\psi_{h,\Delta t}$ be the best approximation  {in the norm of ${H}^{r}_\sigma(\R^+, \widetilde{H}^{-\frac{1}{2}}(\Gamma))$ to the Neumann trace $\partial_n u|_\Gamma$ in ${V}^{p,p}_{\Delta t,h}$} on a quasi-uniform spatial mesh  with $\Delta t \lesssim h$. Then $$\|\partial_n u - \psi_{h,\Delta t}\|_{r, -\frac{1}{2}, \Gamma, \ast} \lesssim \left(\frac{h}{(p+1)^2}\right)^{\frac{1}{2}+\min\{\lambda,0\}{-\varepsilon}} + \left(\frac{h}{p+1}\right)^{\frac{1}{2}+\eta}+  \left(\frac{\Delta t}{p+1}\right)^{\mu+1-r}\ ,$$
 where {$r \in [0,p+1)$}.
\end{theorem}

Theorem \ref{approxtheorem2} follows from the results in Subsections \ref{secvsing} and \ref{secevsing} below, which approximate the leading vertex and edge-vertex singularities. The less singular remainders are approximated as in the previous section.  \textcolor{black}{Theorem \ref{approxtheorem2}} implies a corresponding result for the solutions of the single layer and hypersingular integral equations on the screen:
\begin{corollary}\label{approxcor2} Let $\varepsilon>0$.
a) Let $\phi$ be the solution to the hypersingular integral equation \eqref{hypersingeq} and  $\phi_{h,\Delta t}$ the best approximation {in the norm of ${H}^{r}_\sigma(\R^+, \widetilde{H}^{\frac{1}{2}-s}(\Gamma))$ to $\phi$ in $\widetilde{V}^{p,p}_{\Delta t,h}$} on a quasi-uniform spatial mesh  with $\Delta t \lesssim h$. \textcolor{black}{Assume that the right hand side $g \in {H}^{\alpha}_\sigma(\R^+, \widetilde{H}^{\beta}(\Gamma))$ for some $\alpha,\beta$, so that the regular part $v_0$ \textcolor{black}{belongs to} ${H}^{\mu}_\sigma(\R^+, \widetilde{H}^{\eta}(\Gamma))$  in the singular expansion of $u|_\Gamma$, with $\eta, \mu$ sufficiently large}. Then  $$\|\phi-\phi_{h, \Delta t}\|_{r,\frac{1}{2}-s, \Gamma, \ast} \lesssim \left(\frac{h}{p^2}\right)^{\frac{1}{2}+ \min\{\lambda,0\}+s{-\varepsilon}} + \left(\frac{h}{p}\right)^{-\frac{1}{2}+s+\eta}+  \left(\frac{\Delta t}{p}\right)^{\mu+s-r-\frac{1}{2}}\ ,$$
 where $r \in [0,p)$, $s\in [0,\frac{1}{2}]$.

b) Let $\psi$ be the solution to the single layer integral equation \eqref{dirproblemV} and  $\psi_{h,\Delta t}$ the best approximation  {in the norm of ${H}^{r}_\sigma(\R^+, \widetilde{H}^{-\frac{1}{2}}(\Gamma))$ to $\psi$ in ${V}^{p,p}_{\Delta t,h}$} on a quasi-uniform spatial mesh  with $\Delta t \lesssim h$. \textcolor{black}{Assume that the right hand side $f \in {H}^{\alpha}_\sigma(\R^+, \widetilde{H}^{\beta}(\Gamma))$ for some $\alpha,\beta$, so that the regular part $\psi_0$ \textcolor{black}{belongs to} ${H}^{\mu}_\sigma(\R^+, \widetilde{H}^{\eta}(\Gamma))$  in the singular expansion of $\partial_n u|_\Gamma$, with $\eta, \mu$ sufficiently large}. Then $$\|\psi - \psi_{h,\Delta t}\|_{r, -\frac{1}{2}, \Gamma, \ast} \lesssim \left(\frac{h}{(p+1)^2}\right)^{\frac{1}{2}+\min\{\lambda,0\}{-\varepsilon}} + \left(\frac{h}{p+1}\right)^{\frac{1}{2}+\eta}+  \left(\frac{\Delta t}{p+1}\right)^{\mu+1-r}\ ,$$
 where {$r \in [0,p+1)$}.
\end{corollary}

\subsection{Vertex singularities}\label{secvsing}

To prove Theorem \ref{approxtheorem2} for the Dirichlet trace, we first consider the approximation of the vertex singularities. The regular part is estimated as on the circular screen, and the edge vertex singularities are the content of the following subsection.\\

\textcolor{black}{We recall a key elliptic result for the vertex singularities. Part a) is the content of Theorem 3.6 in \cite{heuer2} and its extension to $hp$ in Theorem 6.1 of \cite{bh08b}, whereas part b) follows from Theorem 3.6 in \cite{heuer3} and its extension to $hp$ in Theorem 5.3 in \cite{bh08}. Note that the singular exponent $\nu$ in these works equals $\frac{1}{2}$ here, and the solution called $u$ in  \cite{bh08} of the equation for the single layer operator  corresponds to $\psi = \partial_n u|_\Gamma$ in our notation.}
\begin{theorem}\label{singapproxbh}
a) Assume $\lambda > 0$ and $p\geq \lambda$. There exists $U^v \in \widetilde{V}_h^p$ such that for all $0 \leq s \leq 1$ $$\|u^v - U^v\|_{s,\Gamma,\ast} \lesssim  \left(\frac{h}{p^2}\right)^{\lambda +1- s-\varepsilon}\ . $$
b) Assume $\lambda > -\frac{1}{2}$ and $p\geq \lambda-1$. There exists $\Psi^v \in V_h^p$ such that for all $-1 \leq s \leq \min\{0,\lambda\}$ $$\|\psi^v - \Psi^v\|_{s,\Gamma,\ast} \lesssim  \left(\frac{h}{(p+1)^2}\right)^{\lambda - s-\varepsilon}\ .$$
\end{theorem}

We now estimate the error of approximating the vertex singularity \eqref{decompositiont} of the Dirichlet trace $u|_\Gamma$ 
\begin{align*}
\|C(t)r^{\lambda} \Phi(\theta) - \Pi_{\Delta t}^{{p}} \widetilde{\Pi}_{x,y}^{{p}} C(t)r^{\lambda} \Phi(\theta)\|_{r,\frac{1}{2}-s,\Gamma,\ast} & \leq \|C(t)r^{\lambda} \Phi(\theta) -  r^{\lambda}  \Phi(\theta) \Pi_{\Delta t}^{{p}} C(t)\|_{r,\frac{1}{2}-s,\Gamma,\ast} \\ & \qquad + \|r^{\lambda} \Phi(\theta) \Pi_{\Delta t}^{{p}} C(t) -  \widetilde{\Pi}_{x,y}^{{p}} r^{\lambda} \Phi(\theta) \Pi_{\Delta t}^{{p}}C(t)\|_{r,\frac{1}{2}-s,\Gamma,\ast} \ .
\end{align*}
For the first term, with Lemma \ref{lemma3.3av} \textcolor{black}{($\frac{1}{2}-s\geq 0$), respectively Lemma \ref{lemma3.3v} ($\frac{1}{2}-s< 0$):}
\begin{align*}
\|C(t)r^{\lambda} \Phi(\theta) -  r^{\lambda}  \Phi(\theta) \Pi_{\Delta t}^{{p}} C(t)\|_{r,\frac{1}{2}-s,\Gamma,\ast}& = \|  r^{\lambda}  \Phi(\theta) (1-\Pi_{\Delta t}^{{p}}) C(t)\|_{r,\frac{1}{2}-s,\Gamma,\ast}\\
& \lesssim \|  r^{\lambda}  \Phi(\theta)\|_{\frac{1}{2}-s,\Gamma, \ast}\| (1-\Pi_{\Delta t}^{{p}}) C(t)\|_{H^{r-s+\frac{1}{2}}_\sigma(\mathbb{R}^+)}\\
& \lesssim \left(\frac{\Delta t}{p}\right)^{\mu+s-r-\frac{1}{2}} \textcolor{black}{\|C(t)\|_{H^\mu_\sigma(\mathbb{R}^+)} }\ .
\end{align*}
For the second term we note
\begin{align*}
\|r^{\lambda} \Phi(\theta) \Pi_{\Delta t}^{{p}}C(t) -  \widetilde{\Pi}_{x,y}^{{p}} r^{\lambda} \Phi(\theta) \Pi_{\Delta t}^{{p}}C(t)\|_{r,\frac{1}{2}-s,\Gamma,\ast} &= \|(1-\widetilde{\Pi}_{x,y}^p) r^{\lambda} \Phi(\theta) \Pi_{\Delta t}^p C(t)\|_{r,\frac{1}{2}-s,\Gamma,\ast}\\
& \lesssim \|(1-\widetilde{\Pi}_{x,y}^p) r^{\lambda} \Phi(\theta)\|_{\frac{1}{2}-s,\Gamma,\ast} \| \Pi_{\Delta t}^p C(t)\|_{H^{r-s+\frac{1}{2}}_\sigma(\mathbb{R}^+)}\ . 
\end{align*}
From Theorem \ref{singapproxbh}a) we conclude
$$\|(1-\widetilde{\Pi}_{x,y}^p) r^{\lambda} \Phi(\theta)\|_{\frac{1}{2}-s,\Gamma,\ast}\lesssim \left(\frac{h}{p^2}\right)^{\lambda +\frac{1}{2}+s-\varepsilon}\ .$$

To estimate the approximation error for the Neumann trace $\partial_n u|_\Gamma$, we note \textcolor{black}{with the expansion from \eqref{decomposition}}
\begin{align*}
\|C'(t)r^{\lambda-1} \Phi'(\theta) - \Pi_{\Delta t}^{{p}} \Pi_{x,y}^{{p}} C'(t)r^{\lambda-1} \Phi'(\theta)\|_{r,-\frac{1}{2},\Gamma,\ast} & \leq \|C'(t)r^{\lambda-1} \Phi'(\theta) -  r^{\lambda-1}  \Phi'(\theta) \Pi_{\Delta t}^{{p}} C'(t)\|_{r,-\frac{1}{2},\Gamma,\ast} \\ & \quad + \|r^{\lambda-1} \Phi'(\theta) \Pi_{\Delta t}^{{p}} C'(t) -  \Pi_{x,y}^{{p}} r^{\lambda-1} \Phi'(\theta) \Pi_{\Delta t}^{{p}}C'(t)\|_{r,-\frac{1}{2},\Gamma,\ast} \ .
\end{align*}
For the first term, with Lemma \ref{lemma3.3v},
\begin{align*}
\|C'(t)r^{\lambda-1} \Phi'(\theta) -  r^{\lambda-1}  \Phi'(\theta) \Pi_{\Delta t}^{{p}} C'(t)\|_{r,-\frac{1}{2},\Gamma,\ast}& = \|  r^{\lambda-1}  \Phi'(\theta) (1-\Pi_{\Delta t}^{{p}}) C'(t)\|_{r,-\frac{1}{2},\Gamma,\ast}\\
& \lesssim \|  r^{\lambda-1}  \Phi'(\theta)\|_{-\frac{1}{2},\Gamma,\ast}\| (1-\Pi_{\Delta t}^{{p}}) C'(t)\|_{H^r_\sigma(\mathbb{R}^+)}\\
& \lesssim \left(\frac{\Delta t}{p+1}\right)^{\mu+1-r}\textcolor{black}{ \|C'(t)\|_{H^{\mu+1}_\sigma(\mathbb{R}^+)} }\ .
\end{align*}
For the second term we note
\begin{align*}
\|r^{\lambda-1} \Phi'(\theta) \Pi_{\Delta t}^{{p}}C'(t) -  \Pi_{x,y}^{{p}} r^{\lambda-1} \Phi'(\theta) \Pi_{\Delta t}^{{p}}C'(t)\|_{r,-\frac{1}{2},\Gamma,\ast} &= \|(1-\Pi_{x,y}^p) r^{\lambda-1} \Phi'(\theta) \Pi_{\Delta t}^p C'(t)\|_{r,-\frac{1}{2},\Gamma,\ast}\\
& \lesssim \|(1-\Pi_{x,y}^p) r^{\lambda-1} \Phi'(\theta)\|_{-\frac{1}{2}, \Gamma, \ast} \| \Pi_{\Delta t}^p C'(t)\|_{H^r_\sigma(\mathbb{R}^+)} \ .
\end{align*}
From Theorem \ref{singapproxbh}b) we conclude
$$\|(1-\Pi_{x,y}^p) r^{\lambda-1} \Phi'(\theta)\|_{-\frac{1}{2},\Gamma,\ast}\lesssim \left(\frac{h}{(p+1)^2}\right)^{\lambda +\frac{1}{2}-\varepsilon}\ .$$

\subsection{Edge-vertex singularities}\label{secevsing}

To conclude the proof of Theorem \ref{approxtheorem2}, it remains to consider the approximation of the edge-vertex singularities.\\

\textcolor{black}{We recall the key elliptic result for the edge-vertex singularities. Part a) is the content of  Theorem 3.5 in \cite{heuer2} and its extension to hp  in Theorem 5.1 of \cite{bh08b}, whereas part b) follows from Theorem 3.4 in \cite{heuer3} and its extension to $hp$ in Theorem 5.1 of \cite{bh08}. }

\begin{theorem}\label{singapproxevbh}
a) Assume $\lambda > 0$ and $p\geq \min\{\lambda, 0\}$. There exists $U^{ev} \in \widetilde{V}_h^p$ such that for all $0 \leq s \leq \min\{\lambda +1,1\}$ 
$$\|u^{ev} - U^{ev}\|_{s, \Gamma, \ast} \lesssim  \left(\frac{h}{p^2}\right)^{1+\min\{\lambda,0\}- s-\varepsilon}\ .$$
b) Assume $\lambda > -\frac{1}{2}$ and \textcolor{black}{$p\geq \min\{\lambda-1,0\}$.} There exists $\Psi^{ev} \in V_h^p$ such that for all $-1 \leq s \leq \min\{0,\lambda\}$ 
$$\|\psi^{ev} - \Psi^{ev}\|_{s, \Gamma, \ast} \lesssim  \left(\frac{h}{(p+1)^2}\right)^{\min\{\lambda,0\} - s-\varepsilon}\ .$$
\end{theorem}

Concerning the edge-vertex singularities of the Dirichlet trace $u|_\Gamma$, we restrict ourselves to $u^{ev}_1$ in \eqref{decompositiont}. We bound  the corresponding approximation error by
\begin{align*}&\|u^{ev}_1 - \widetilde{\Pi}_{h,x}^{{p}} \Pi_{\Delta t}^{{p}} u^{ev}_1\|_{r, \frac{1}{2}-s, \Gamma, \ast} \leq \|C_1(t)\tilde{\chi}(\theta)\beta_{1}(r)(\sin(\theta))^{\frac{1}{2}} - \Pi_{\Delta t}^{{p}} \widetilde{\Pi}_{h,x}^{p} C_1(t)\tilde{\chi}(\theta)\beta_{1}(r)(\sin(\theta))^{\frac{1}{2}}\|_{r, \frac{1}{2}-s, \Gamma,\ast} \\ & \leq \|(1-\Pi_{\Delta t}^{{p}}) C_1(t)\tilde{\chi}(\theta)\beta_{1}(r)(\sin(\theta))^{\frac{1}{2}}\|_{r, \frac{1}{2}-s, \Gamma,\ast}+\|(1- \widetilde{\Pi}_{h,x}^{p}) \Pi_{\Delta t}^{{p}} C_1(t)\tilde{\chi}(\theta)\beta_{1}(r)(\sin(\theta))^{\frac{1}{2}}\|_{r, \frac{1}{2}-s, \Gamma,\ast} \\ &
\lesssim \|C_1(t) -\Pi_{\Delta t}^{{p}} C_1(t)\|_{H^{r-s+\frac{1}{2}}_\sigma(\mathbb{R}^+)} \|\tilde{\chi}(\theta)\beta_{1}(r)(\sin(\theta))^{\frac{1}{2}}\|_{\frac{1}{2}-s, \Gamma, \ast} \\ &\ \ \ \ +  \|(1- \widetilde{\Pi}_{h,x}^{p})\tilde{\chi}(\theta)\beta_{1}(r)(\sin(\theta))^{\frac{1}{2}}\|_{\frac{1}{2}-s, \Gamma, \ast}\|C_1(t)\|_{H^{r-s+\frac{1}{2}}_\sigma(\mathbb{R}^+)} \ .
\end{align*}
Here, for the first term we have used Lemma \ref{lemma3.3av}, respectively Lemma \ref{lemma3.3v}. We note that the first term is bounded by $$\|C_1(t) -\Pi_{\Delta t}^{{p}} C_1(t)\|_{H^{r+s-\frac{1}{2}}_\sigma(\mathbb{R}^+)}\lesssim \left(\frac{ \Delta t}{p}\right)^{\mu+s-r-\frac{1}{2}} \|C_1(t)\|_{H^\mu_\sigma(\mathbb{R}^+)}\ .$$
From Theorem \ref{singapproxevbh}a) we have  $$\|(1- \widetilde{\Pi}_{h,x}^{p})\tilde{\chi}(\theta)\beta_{1}(r)(\sin(\theta))^{\frac{1}{2}}\|_{\frac{1}{2}-s, \Gamma, \ast} \lesssim \left(\frac{h}{p^2}\right)^{\min\{\lambda,0\} +\frac{1}{2}+s-\varepsilon}\ .$$

For the edge-vertex singularities of the Neumann trace $\partial_n u|_\Gamma$, we consider $\psi^{ev}_1$ in \eqref{decomposition} and estimate the approximation error as follows:
\begin{align*}&\|\psi^{ev}_1 - \Pi_{h,x}^{{p}} \Pi_{\Delta t}^{{p}} \psi^{ev}_1\|_{r, -\frac{1}{2}, \Gamma, \ast} \leq \|C_1'(t)\tilde{\chi}(\theta)\beta_{1}'(r)r^{-1}(\sin(\theta))^{-\frac{1}{2}} - \Pi_{\Delta t}^{{p}} \Pi_{h,x}^{p} C_1'(t)\tilde{\chi}(\theta)\beta_{1}'(r)r^{-1}(\sin(\theta))^{-\frac{1}{2}}\|_{r, -\frac{1}{2}, \Gamma,\ast} \\ & \leq \|(1-\Pi_{\Delta t}^{{p}}) C_1'(t)\tilde{\chi}(\theta)\beta_{1}'(r)r^{-1}(\sin(\theta))^{-\frac{1}{2}}\|_{r, -\frac{1}{2}, \Gamma,\ast}+\|(1- \Pi_{h,x}^{p}) \Pi_{\Delta t}^{{p}} C_1'(t)\tilde{\chi}(\theta)\beta_{1}'(r)r^{-1}(\sin(\theta))^{-\frac{1}{2}}\|_{r, -\frac{1}{2}, \Gamma,\ast} \\ &
\lesssim \|C_1'(t) -\Pi_{\Delta t}^{{p}} C_1'(t)\|_{H^r_\sigma(\mathbb{R}^+)} \|\tilde{\chi}(\theta)\beta_{1}'(r)r^{-1}(\sin(\theta))^{-\frac{1}{2}}\|_{-\frac{1}{2},\Gamma,\ast}  \\ & \ \ \ \ \ +  \|(1- \Pi_{h,x}^{p})\tilde{\chi}(\theta)\beta_{1}'(r)r^{-1}(\sin(\theta))^{-\frac{1}{2}}\|_{-\frac{1}{2}, \Gamma, \ast}\|C_1'(t)\|_{H^r_\sigma(\mathbb{R}^+)} \ .
\end{align*}
Here, for the first term we have used Lemma \ref{lemma3.3v}. We note that the first term is bounded by $$\|C_1'(t) -\Pi_{\Delta t}^{{p}} C_1'(t)\|_{H^r_\sigma(\mathbb{R}^+)}\lesssim \left(\frac{ \Delta t}{p+1}\right)^{\mu+1-r} \|C_1'(t)\|_{H^{\mu+1}_\sigma(\mathbb{R}^+)}\ .$$
From Theorem \ref{singapproxevbh}b) we have  $$\|(1- \Pi_{h,x}^{p})\tilde{\chi}(\theta)\beta_{1}'(r)r^{-1}(\sin(\theta))^{-\frac{1}{2}}\|_{-\frac{1}{2}, \Gamma, \ast} \lesssim \left(\frac{h}{(p+1)^2}\right)^{\min\{\lambda,0\} +\frac{1}{2}-\varepsilon}\ .$$

\subsection{Singularities for polyhedral domains and approximation}\label{polhed}

The screen in the previous sections was the degenerate case of a polyhedral domain with opening angle $2\pi$ of the wedges, which leads to the strongest singularities. In general, for polyhedral domain with edge opening angles $<2\pi$ the leading edge exponents of the solution $u$ in \eqref{singexp} with either Dirichlet or Neumann conditions are given by $\nu_{1,B}= \frac{\pi}{\alpha}$, where $\alpha$ is the opening angle of the wedge. Schwab and Suri \cite{schwabs} provide $p$-explicit approximation results for the Dirichlet trace. We state the general approximation theorem for the elliptic case, which follows from the results of \cite{schwabs} and (for the Neumann trace) the stronger results of \cite{heuer3}, see Theorems \ref{singapproxbh} and \ref{singapproxevbh} above. 

\begin{theorem}\label{icosrate}
a) There exists a function $u_{hp}\in \widetilde{V}_h^p$ such that for $s \in [0,1]$:
$$\|u-u_{hp}\|_{s,\Gamma,\ast} \lesssim \max\left\{\frac{h^{k-s}}{p^{k-s}}, \frac{h^{\nu-s+\frac{1}{2}}}{p^{2\nu-2s+1}},\frac{h^{\lambda-s+1-\varepsilon}}{p^{2\lambda-2s+2-2\varepsilon}}\right\}\ .$$
Here $v_{0} \in H^{k}(\Gamma)$.\\
b) There exists a function $\psi_{hp} \in V_h^p$ such that:
$$\|\partial_n u-\psi_{hp}\|_{-\frac{1}{2},\Gamma,\ast} \lesssim \max\left\{\frac{h^{k-\frac{1}{2}}}{(p+1)^{k-\frac{1}{2}}}, \frac{h^{\nu}}{(p+1)^{2\nu}},\frac{h^{\lambda+\frac{1}{2}-\varepsilon}}{(p+1)^{2\lambda+1-2\varepsilon}}\right\}\ .$$
Here $\psi_{0} \in H^{k}(\Gamma)$.
\end{theorem}

Here the second term in the maximum is the approximation error of the edge singular function, while the third is the approximation error of the vertex singular function. The first term in the maximum is due to the approximation of the remainder of the asymptotic expansion.\\

Also in the time dependent case of the wave equation, the edge singularities dominate, except in domains with sharp reentrant corners \cite{petersdorff2}. For the Dirichlet and Neumann traces  the exponents are the same as in the time independent case. Following the above analysis for the screen, by using the estimates for the approximation error of the time-independent singular functions at the vertices and edges from the proof of Theorem \ref{icosrate}, one can show: Let $u$ be the  solution to the homogeneous wave equation with inhomogeneous Neumann boundary conditions $\partial_n u|_\Gamma = g$, \textcolor{black}{with $g \in {H}^{\alpha}_\sigma(\R^+, \widetilde{H}^{\beta}(\Gamma))$ for some $\alpha,\beta$, so that the regular part $v_0$ \textcolor{black}{belongs to} ${H}^{\mu}_\sigma(\R^+, \widetilde{H}^{\eta}(\Gamma))$  in the singular expansion of $u|_\Gamma$, with $\eta, \mu$ sufficiently large}. If $\phi_{h,\Delta t}$ is the best approximation {in the norm of ${H}^{r}_\sigma(\R^+, \widetilde{H}^{\frac{1}{2}-s}(\Gamma))$ to the Dirichlet trace $u|_\Gamma$ in $\widetilde{V}^{p,p}_{\Delta t,h}$} on a quasi-uniform spatial mesh with $\Delta t \lesssim h$, then for every $\varepsilon>0$ $$\|u-\phi_{h, \Delta t}\|_{r,\frac{1}{2}-s, \Gamma, \ast} \lesssim \max\left\{\frac{h^{k-s}}{p^{k-s}}, \frac{h^{\nu-s+\frac{1}{2}}}{p^{2\nu-2s+1}},\frac{h^{\lambda-s+1-\varepsilon}}{p^{2\lambda-2s+2-2\varepsilon}}\right\} + \left(\frac{\Delta t}{p}\right)^{\mu+s-r-\frac{1}{2}}\ .$$
 Here {$r \in [0,p)$}.\\

\textcolor{black}{This result generalizes \textcolor{black}{part a) of Theorem \ref{approxtheorem1} and Theorem \ref{approxtheorem2}} to polyhedral domains instead of flat screens, where $\lambda>0$ and $\nu=\frac{1}{2}$.}\\

Similarly, let $u$ be the  solution to the homogeneous wave equation with inhomogeneous Dirichlet boundary conditions $u|_\Gamma = g$, \textcolor{black}{with $g \in {H}^{\alpha}_\sigma(\R^+, \widetilde{H}^{\beta}(\Gamma))$ for some $\alpha,\beta$, so that the regular part $\psi_0$ \textcolor{black}{belongs to} ${H}^{\mu}_\sigma(\R^+, \widetilde{H}^{\eta}(\Gamma))$  in the singular expansion of $\partial_n u|_\Gamma$, with $\eta, \mu$ sufficiently large}. If $\psi_{h,\Delta t}$ is the best approximation  {in the norm of ${H}^{r}_\sigma(\R^+, \widetilde{H}^{-\frac{1}{2}}(\Gamma))$ to the Neumann trace $\partial_n u|_\Gamma$ in ${V}^{p,p}_{\Delta t,h}$} on a quasi-uniform spatial mesh  with $\Delta t \lesssim h$, then for every $\varepsilon>0$ $$\|\partial_n u - \psi_{h,\Delta t}\|_{r, -\frac{1}{2}, \Gamma, \ast} \lesssim \max\left\{\frac{h^{k-\frac{1}{2}}}{(p+1)^{k-\frac{1}{2}}}, \frac{h^{\nu}}{(p+1)^{2\nu}},\frac{h^{\lambda+\frac{1}{2}-\varepsilon}}{(p+1)^{2\lambda+1-2\varepsilon}}\right\}+ \left(\frac{\Delta t}{p+1}\right)^{\mu+1-r}\ .$$
Here {$r \in [0,p+1)$}.\\

\textcolor{black}{This estimate generalizes \textcolor{black}{part b) of Theorem \ref{approxtheorem1} and Theorem \ref{approxtheorem2}} to polyhedral domains.}\\

\textcolor{black}{Similar to Corollary \ref{approxcor1} for the circular screen, respectively Corollary \ref{approxcor2} for the polygonal screen, also for a polyhedral domain the approximation rates for the Dirichlet and Neumann traces  translate into approximation rates  for appropriate boundary integral equations:} $W\phi = (\frac{1}{2}-K')g$ for the Neumann problem, respectively $V \psi = (\frac{1}{2}-K)f$ for the Dirichlet problem.

\section{Numerical experiments}\label{experiments}

\subsection{Implementation of single layer operator}

On the left hand side of \eqref{weakformhp}, we use  ansatz, respectively test  functions $$\psi_{\Delta t, h}(t,x)=\sum_{m=1}^{N_t} \sum_{i=1}^{N_s}c_m^i\gamma_{\Delta t}^m(t){\psi_h^i(x)  \in V^{p,p}_{h,\Delta t}}, \quad \Psi^{n,l}(t,x)=\gamma_{\Delta t}^{n}(t){\psi_h^l(x)}  \in V^{p,p}_{h,\Delta t}$$ to obtain for the single layer potential:
\begin{align*}
\int_0^\infty \int_\Gamma (V \psi_{\Delta t,h}) \dot{\gamma}^n_{\Delta t}{\psi_h^l} dx dt 
&= \sum_{m,i}c_m^i\frac{1}{4\pi} \int_0^\infty \int_{\Gamma \times \Gamma}\frac{1}{|x-y|}\gamma_{\Delta t}^m(t-|x-y|)
{\psi_h^i}(y)\dot{\gamma}^n_{\Delta t}(t){\psi_h^l}(x) dx dy dt\\
&= \sum_{m,i}c_m^i\frac{1}{4\pi} \int_{\Gamma \times \Gamma}\frac{{\psi_h^i}(y){\psi_h^l}(x)}{|x-y|}\int_0^\infty\gamma_{\Delta t}^m(t-|x-y|)\dot{\gamma}^n_{\Delta t}(t)\ dt\ dx dy
\end{align*}
\noindent for all $n=1, ..., N_t$ and $l=1, ..., N_s$. Here, we use a dot to denote the time derivative.

For example, for piecewise linear basis functions, $p=1$ in space and time, a calculation of the time integral shows: 
\begin{align*}
&\int_0^\infty\gamma_{\Delta t}^m(t-|x-y|)\dot{\gamma}^n_{\Delta t}(t)\ dt\\
&= - \frac{1}{2 (\Delta t)^2} (t_{n-m+2}-|x-y|)^2 \chi_{E_{n-m+1}}(x,y) \\
&\qquad +\left(\frac{1}{(\Delta t)^2} (t_{n-m+1}-|x-y|)^2 +\frac{1}{2(\Delta t)^2} (t_{n-m}-|x-y|)^2 -1\right)\chi_{E_{n-m}}(x,y) \\
&\qquad-\left(\frac{1}{(\Delta t)^2} (t_{n-m-1}-|x-y|)^2 +\frac{1}{2(\Delta t)^2} (t_{n-m}-|x-y|)^2 -1\right)\chi_{E_{n-m-1}}(x,y) \\
&\qquad+\frac{1}{2(\Delta t)^2} (t_{n-m-2}-|x-y|)^2 \chi_{E_{n-m-2}}(x,y) \ ,
\end{align*}
with $$E_l = \lbrace (x,y) \in \Gamma \times \Gamma : t_{l} \leq |x-y| \leq t_{l+1}\}\ .$$ Formulas for higher polynomial degree may be found in \cite{stephan2008transient}. After the time integral is evaluated analytically, the spatial integrals are approximated using a composite $hp$-graded quadrature \cite{gimperleinreview}.\\

The Galerkin discretization leads to a block--lower--Hessenberg system of equations, see Figure \ref{system}. Here the blocks $V^{l}$ correspond to the matrix with entries
$$V^{n-m}_{il} = \int_0^\infty \int_\Gamma (V \gamma^m_{\Delta t}{\psi_h^i}) \dot{\gamma}^n_{\Delta t}{\psi_h^l} dx dt\ .$$
The system can be solved with an approximate time stepping scheme, respectively a space-time preconditioned GMRES method \cite{dsprec}.

 Note that the common, but non-conforming MOT time stepping schemes are based on piecewise constant test functions in time. Then $E_{l+1}$ does not contribute to the matrix entries of $V$, so the block $V^{-1}=0$, and one obtains a block--lower--triangular system of equations. 

\begin{figure}[htbp]
 \centering
 \includegraphics[height=5.2cm]{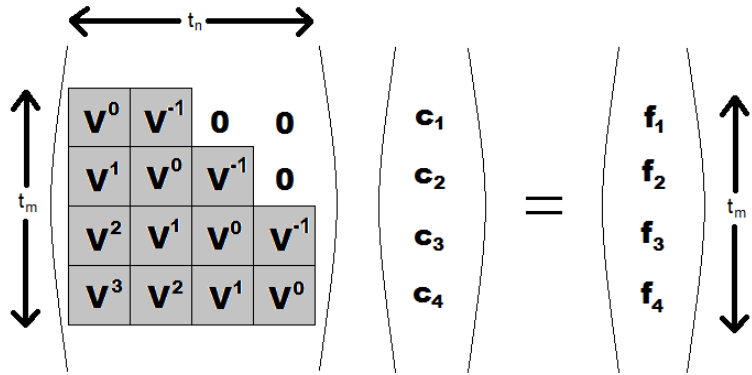}
 \caption{Full space-time system for the $hp$-version of the TDBEM.}
 \label{system}
\end{figure}

\subsection{Wave equation outside a screen}\label{experimentsscr}
 
\begin{example}\label{example0p}
Using the discretization by piecewise polynomials of degree $p$ described above, we compute the solution to the integral equation $V\psi=f$ on $\mathbb{R}_t^+ \times \Gamma$, with the square screen $\Gamma = \{(x,y,0) : -\frac{1}{2} \leq x,y\leq \frac{1}{2}\}$ depicted in Figure \ref{screenmesh}. We use a discretization with $8$ triangles and $9$ nodes in space, a time step $\Delta t = 0.5$, respectively $1.0$, and study the convergence of the numerical solution as the polynomial degree is increased. \textcolor{black}{We compute the solution $\psi_{p}$ up to time $T=4$ and compare the error in the energy norm $|\langle V \psi_{p}, \partial_t \psi_{p} \rangle_{[0,T]\times \Gamma}|^{1/2} = (\int_{0}^{T} \int_{\Gamma} (V  \psi_{p}) \partial_{t} \psi_{p})^{1/2}$ for various right hand sides}. 
\end{example}
\begin{figure}
  \centering
  \includegraphics[width=5cm]{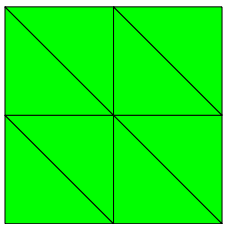}
   \caption{Screen mesh with 8 triangular elements and 9 nodes.}\label{screenmesh}
\end{figure}

\begin{figure}
  \centering
  \hspace*{-1.0cm}
   \includegraphics[width=5.8cm]{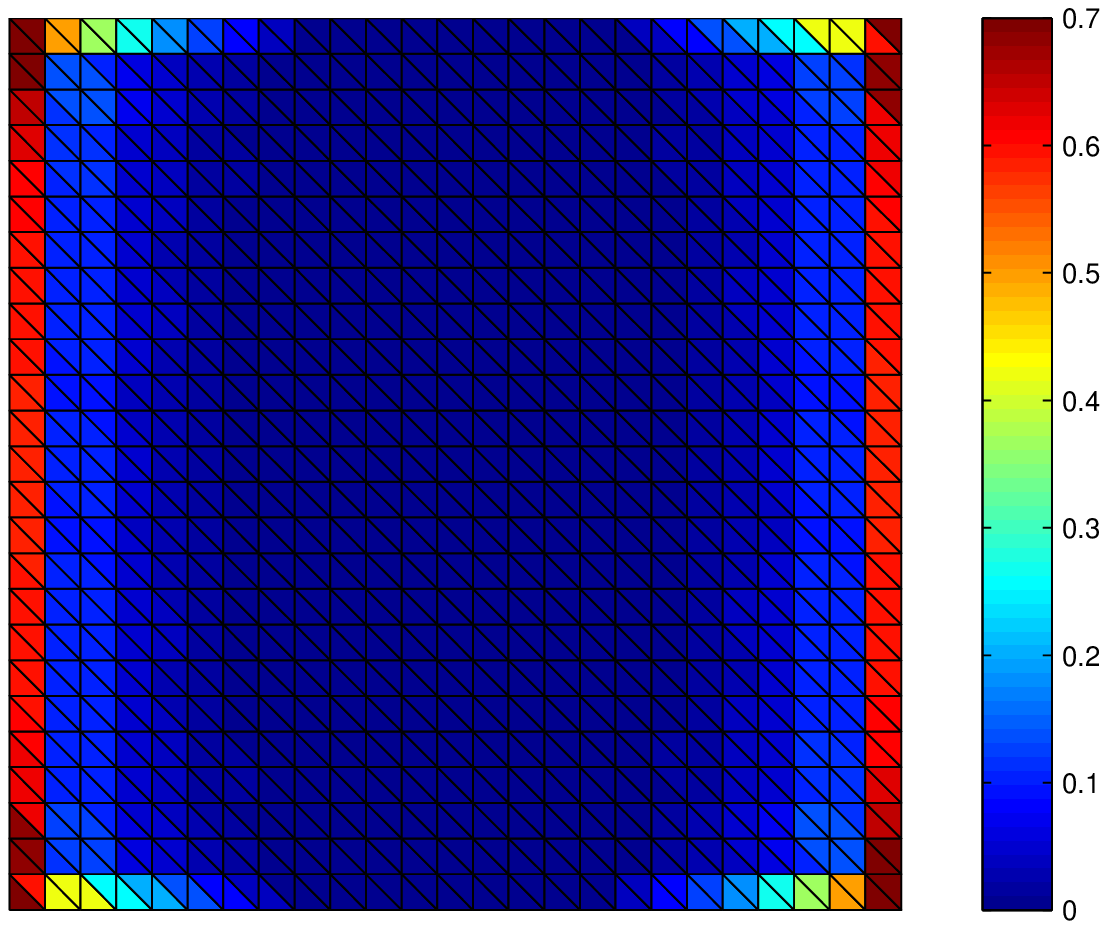}
    \includegraphics[width=5.8cm]{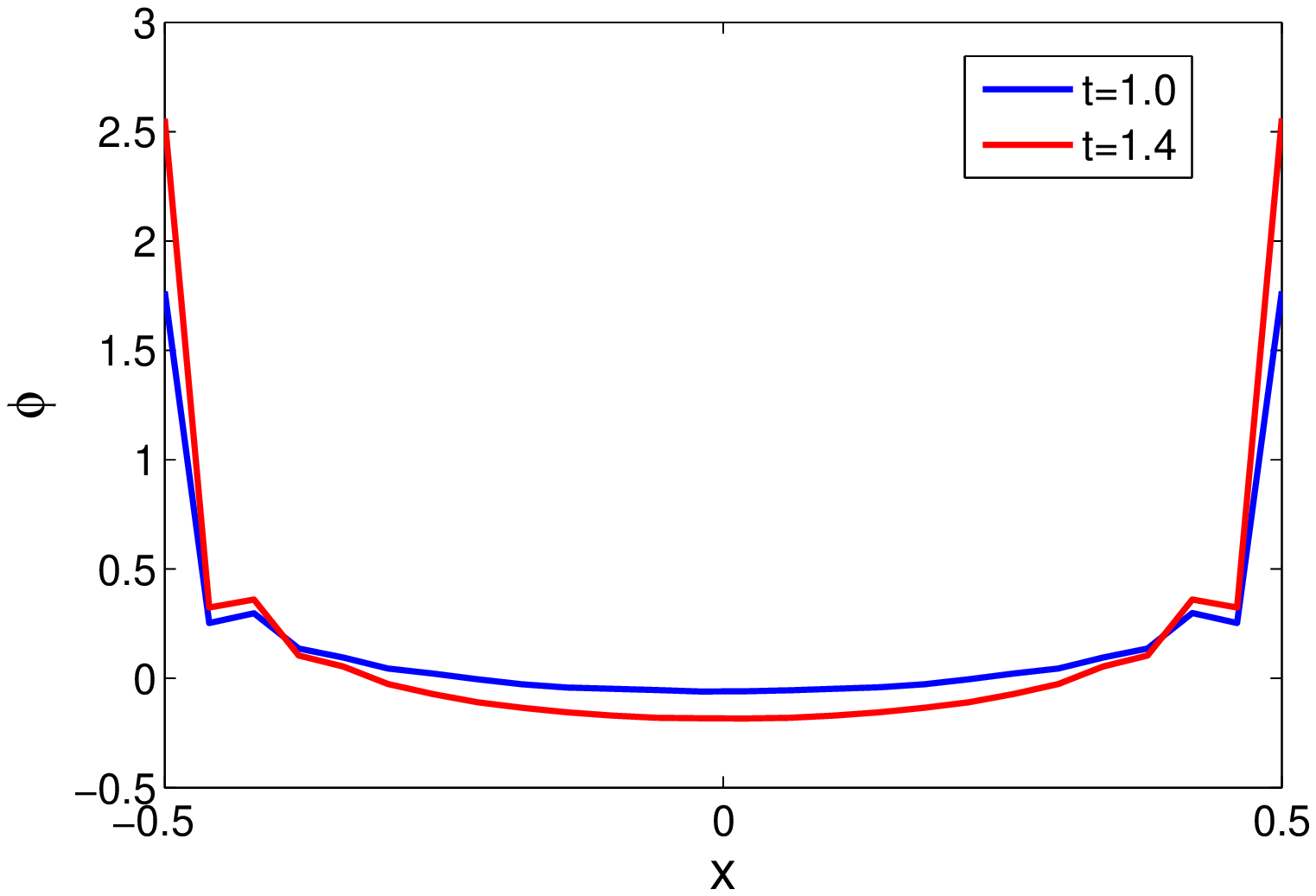}
    \caption{Density $\psi$ computed by $h$-method on a uniform mesh with $1250$ triangles for $f_1$ (left), cross section $y=0$ at $t=1.0, 1.4$ (right).}\label{screensoln}
\end{figure}

From \cite{graded}, the convergence rate in energy norm of the uniform $h$-method on the screen is $0.5$ as $h$ tends to $0$. A cross section at $y=0$  of the solution for the right hand side $$f_1(t,(x,y,z)^{T})=\sin^5(t)x^2$$ is shown in Figure \ref{screensoln}, for a uniform triangulation of $\Gamma$ with $1250$ triangles at times $t=1.0$ and $1.4$. The cross section shows the edge singularities of the solution, as well as unphysical oscillations as numerical errors near the boundary. It indicates the  difficulty of approximating the singularities numerically. 

\begin{figure}
  \centering
  \includegraphics[width=11cm]{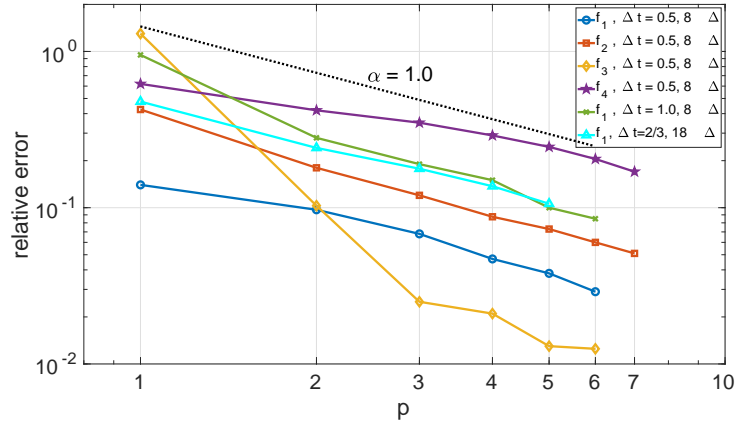}
   \caption{Relative error in energy norm for the single-layer equation on a square screen, Example \ref{example0p}.}\label{screenplots}
\end{figure}\

\begin{figure}
  \centering
  \includegraphics[width=7cm]{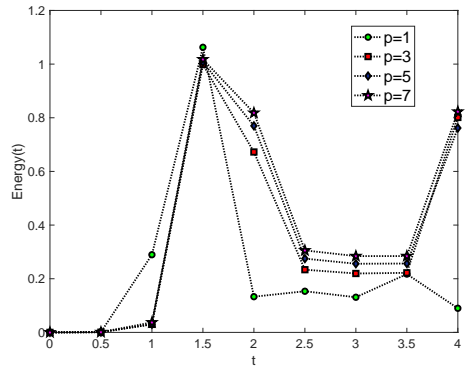}
   \caption{Energy as a function of time for time-singular $f_4$, Example \ref{example0p}.}\label{screensingular}
\end{figure}\

For this  right hand side $f_1$, Figure \ref{screenplots} depicts the convergence in energy norm of a $p$-method up to polynomial degree $p=6$ in space and time. The empirical convergence rate for $\Delta t = 0.5$ (blue dots) is approximately $1.21$. For $\Delta t = 1.0$ the convergence rate is $1.18$ (yellow crosses). The results reflect the expected doubling of the convergence rate for the $p$-method, compared to the $h$-method.

The results are confirmed for plane-wave right hand sides at low frequencies. For the right hand side $$f_2(t,(x,y,z)^{T})=\exp(-2/t^2)cos(\omega t - {k} (x,y,z)^{T})\ ,$$ with ${k} = (2, 0.5, 0.1)$ and $\omega = |{k}|$, Figure \ref{screenplots} (red squares) shows the convergence in energy norm of the $p$-version with rate $1.02$ up to $p=7$, for $\Delta t = 0.5$. For the higher-frequency wave $$f_3(t,(x,y,z)^{T})=\exp(-2/t^2)cos(\omega t - {k} (x,y,z)^{T})\ ,$$ with ${k}= (6, 0.5, 0.1)$ and $\omega = |{k}|$, piecewise linear or quadratic polynomials provide a poor approximation, as shown in Figure  \ref{screenplots} (black diamonds) when $\Delta t = 0.5$. At higher $p$ the convergence rate becomes approximately $1.01$, in agreement with the results for $f_1$ and $f_2$. 

As a last right hand side, a source which is nonsmooth in time is considered, $$f_4(t,(x,y,z)^{T})=\sin^5(t) |1-t|^\alpha \cos({k}\cdot (x,y,z))\ ,$$ with $\alpha=\frac{1}{2}$ and ${k} = (6, 0.5, 0.1)$. Note the square-root singularity in time in this right hand side.  Figure \ref{screensingular} shows the ``energy'' $E(t) = \frac{1}{2}\langle V \psi, \partial_t \psi\rangle_{[0,t]\times \Gamma}- \langle f, \partial_t \psi\rangle_{[0,t]\times \Gamma}$ as a function of time at multiples of the time step $\Delta t = 0.5$, for $p=1,3,5,7$. While the solutions for different $p$ closely agree for short times, after the kink of the right hand side at $t=1$ only higher polynomial degrees $p$ provide similar approximations. The convergence rate in energy norm here is $0.78$, see Figure \ref{screenplots} (green stars), less than for $f_1$, $f_2$ and $f_3$.\\

\textcolor{black}{A final computation  discretizes the screen with $18$ triangles and $16$ nodes in space, and uses a time step $\Delta t = \frac{2}{3}$. The  numerical solution is considered up to $T=3.33333$ as the polynomial degree $p$ is increased. 
The error in the energy norm goes to zero at rate $1.01$, as depicted in Figure \ref{screenplots} and is smaller than the corresponding error on $9$ triangles with time step $\Delta t = 1.0$.}

The second example studies the $h$-method for different polynomial degrees $p$.
\textcolor{black}{
\begin{example}\label{screenhversionp123}
Using the discretization by piecewise polynomials of degree $p=1,2,3$ from Example \ref{example0p}, we compute the solution to the integral equation $V\psi=f$ on $\mathbb{R}_t^+ \times \Gamma$, with the square screen $\Gamma = \{(x,y,0) : -\frac{1}{2} \leq x,y\leq \frac{1}{2}\}$ as above. We  study the convergence of the numerical solution in the square  of the energy norm at time $T=2$ as the mesh is refined. As benchmark, we use solutions for $\Delta t = 0.166$ and $288$ triangles for $p=1$ and  $\Delta t = 0.25$ and $128$ triangles for $p=2,3$. 
\end{example}
}
\textcolor{black}{
Figure \ref{screenplotshversion}  shows the $h$-version for $f_1$ and $f_4$ from above. We observe that for $p=3$ we obtain for $f_{1}$ a convergence rate of $0.52$ and for $f_{4}$ a rate of $0.468$, which is in a good agreement with the expected value of $0.5$. For $p=2$ and $f_{4}$ we get a rate of $0.48$, and for $f_{1}$ we get a rate of $0.425$.
The kink in the last point can be explained by remarking that the refinement with $72$ triangles and $\Delta t=1/3$ is close to the benchmark for $p=2,3$. We get the same kink for $p=1$, where the refinement is $200$ triangles with $\Delta t = 0.2$. For $p=1$ and $f_{4}$ we get a rate of $0.73$, where the middle part of $p=1$ for $f_{1}$ gives $0.53$. The achieved convergence rates for $p=1$ are due to the preasymptotic region.  The convergence rates correspond to a rate of $0.5$ in terms of $h$.
}
\textcolor{black}{
The numerical results underline our theoretical conclusions from the analysis in this article: The convergence rate of the $h$-method is half the convergence rate of the $p$-method. It is independent of the polynomial degree.
}

\begin{figure}
  \centering
  \includegraphics[width=12cm]{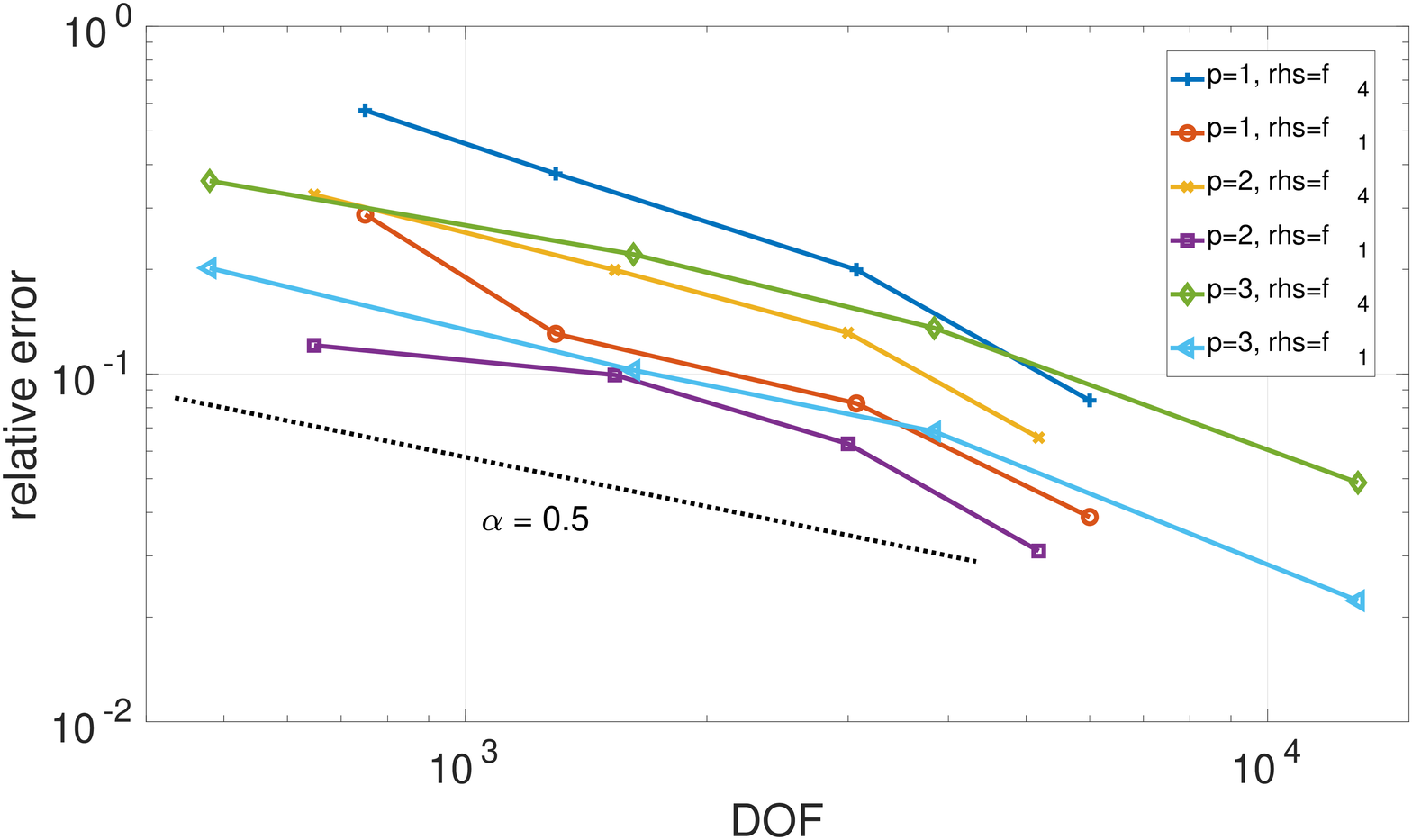}
   \caption{Relative error in the square  of the energy norm for the single-layer equation on a square screen, $h$-version for $p=1,2,3$, Example \ref{screenhversionp123}.}\label{screenplotshversion}
\end{figure}\

\subsection{Wave equation outside an icosahedron}\label{experimentsico}
 
\begin{example}\label{example1p}
Using the discretization by piecewise polynomials of degree $p$ described above, we compute the solution to the integral equation $V\psi=f$ on $\mathbb{R}_t^+ \times \Gamma$, for the icosahedron $\Gamma$  depicted in Figure \ref{icosmesh1}. We use the discretization given by the $20$ triangular faces of the icosahedron with $12$ vertices and a time step $\Delta t = 0.5$. The convergence of the numerical solution is studied as the polynomial degree is increased.  Different right hand sides $f$ are considered. We compute the solution for long times up to $T=11$ and compare to an extrapolated benchmark energy as in Example \ref{example0p}. Based on Section \ref{polhed} for the direct integral equation $V\psi=(\frac{1}{2}-K) f$ one expects  a convergence rate for the $p$-version of $1.62$, dominated by the edge singularities. 
\end{example}
\begin{figure}
  \centering
  \includegraphics[width=7cm]{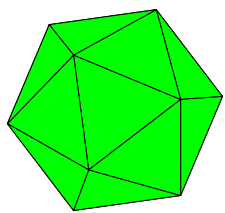}
   \caption{Icosahedron with 20 triangles and 12 vertices.}\label{icosmesh1}
\end{figure}
\begin{figure}
  \centering
  \includegraphics[width=5cm]{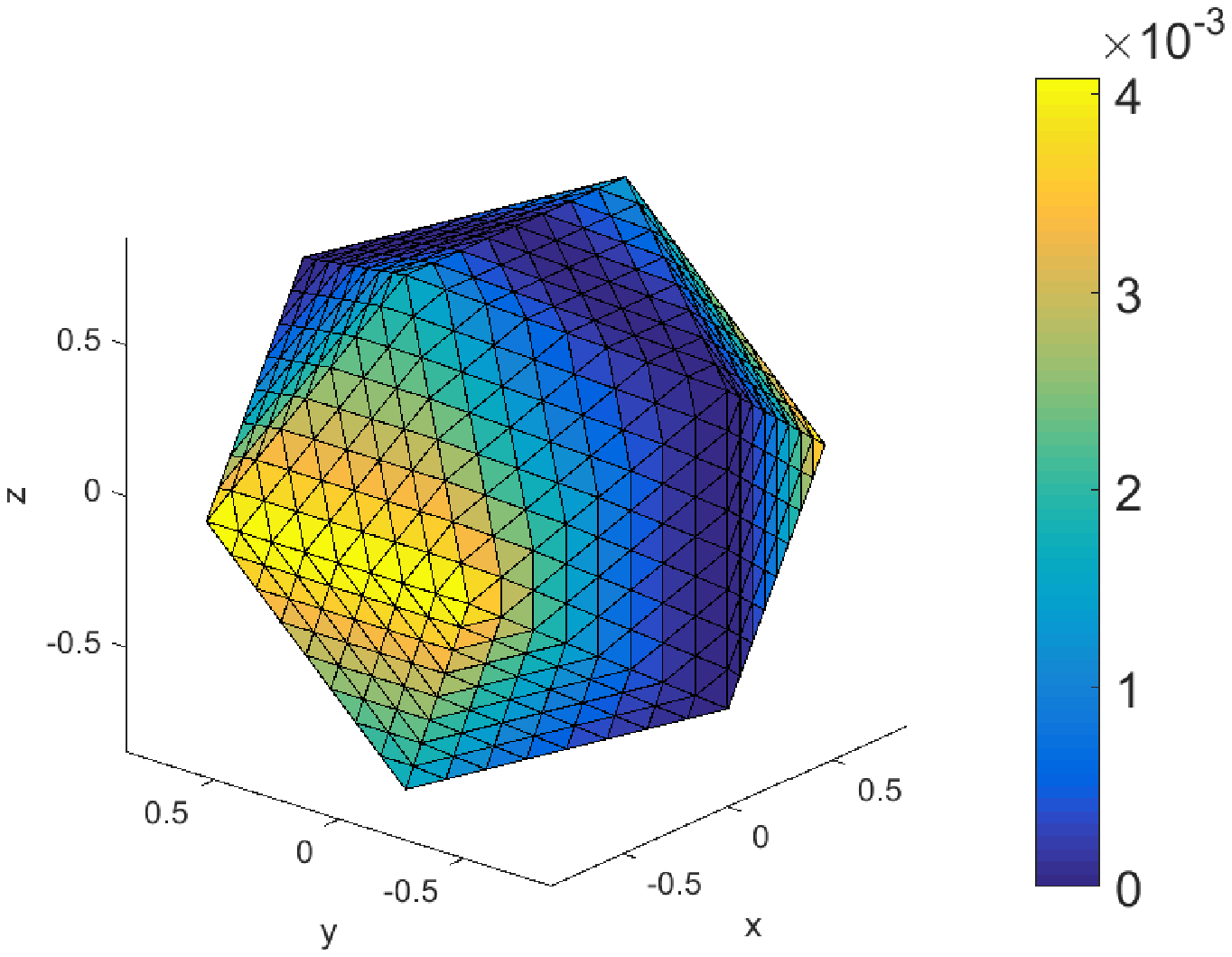}
   \includegraphics[width=5cm]{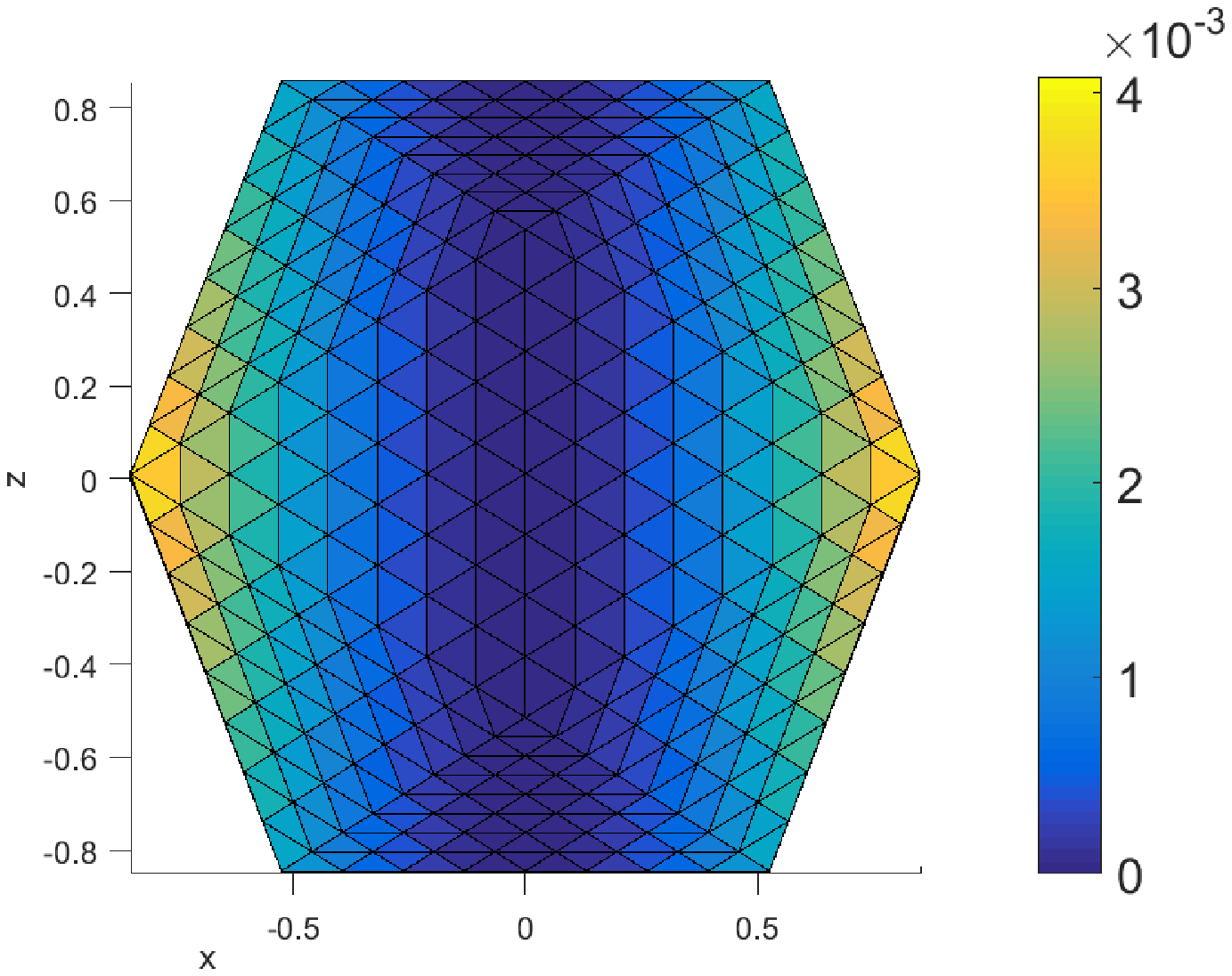}
   \caption{Density $\psi$ computed by $h$-method on a uniform mesh with $1280$ triangles for $f_1$.}\label{icosdens1}
\end{figure}

A picture of the smooth solution at time $t=0.5$ for the right hand side $$f_1(t,(x,y,z)^{T})=\sin^5(t)x^2$$ is shown in Figure \ref{icosdens1}, computed using an $h$-method on a uniform triangulation of $\Gamma$ with $1280$ triangles and time step $\Delta t = 0.1$.

\begin{figure}
  \centering
 \includegraphics[width=9cm]{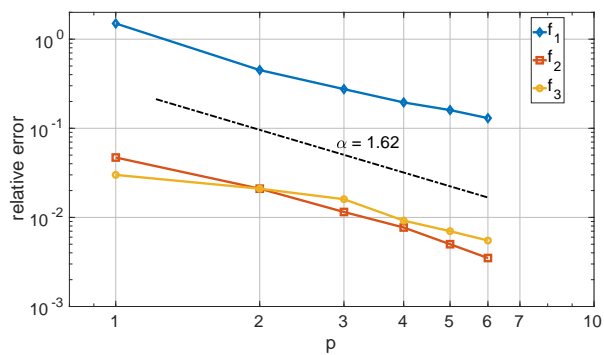}
   \caption{Relative error in energy norm of $p$-method for the single-layer equation on icosahedron, Example \ref{example1p}.}\label{icosconvergence}
\end{figure}

Figure \ref{icosconvergence} shows the convergence of the $p$-method in the energy norm for the right hand side $f_1$ from above (blue circles). The empirical convergence rate is $1.46$ as the polynomial degree $p$ is increased. Figure \ref{sauterlong} shows the possibility of long-time simulations and plots the energy of the numerical solution with $p=6$ as a function up to time $t=11$ at multiples of the time step $\Delta t = 0.5$.  Figure \ref{icosdifference} depicts the difference $|E_6(t) - E_p(t)|$ between the energy of the $p$-method solution for $p=6$ and the numerical solutions for $p=1,2,\dots, 5$.  The error remains stable over the time interval, reflecting the space-time variational discretization used \cite{gimperleintyre}.  

A second right hand side investigates  a plane-wave $$f_2(t,(x,y,z)^{T})=\exp(-2/t^2)\cos(\omega t - {k} (x,y,z)^{T})\ ,$$ with ${k} = (3, 0.5, 0.1)$ and $\omega = |{k}|$.  The convergence rate in this case is approximately $1.61$, see Figure \ref{icosconvergence}, in agreement with the analysis and slightly higher than for $f_1$. 

Finally, a right hand side with a singularity in space is considered, $$f_3(t,(x,y,z)^{T})=\sin^5(t)|\sin({k}(x,y,z)^{T})|^\alpha\ ,$$ $\alpha=\frac{1}{2}$ and ${k} = (2, 0.5, 0.1)$.  The convergence rate here is lower, $1.22$. Note that the solution $\psi$ has a singularity in space on the lines ${k}{x} = k\pi$, $k \in \mathbb{Z}$, similar to the edge singularities in Example \ref{example0p}. The convergence rate in Figure \ref{icosconvergence} is therefore reduced to values closer to those seen for screen problems in Example \ref{example0p}.

\begin{figure}
  \centering
  \includegraphics[width=7cm]{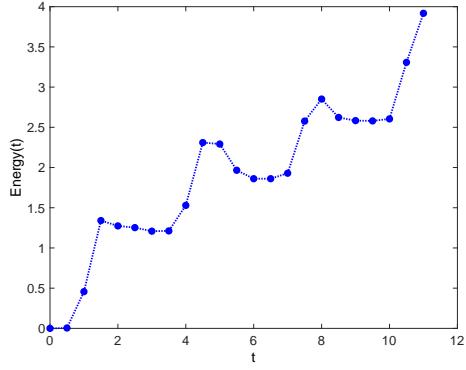}
   \caption{Energy as a function of time up to $t=11$ for the right hand side $f_1$, Example \ref{example1p}.}\label{sauterlong}
\end{figure}\

\begin{figure}
  \centering
  \includegraphics[width=7cm]{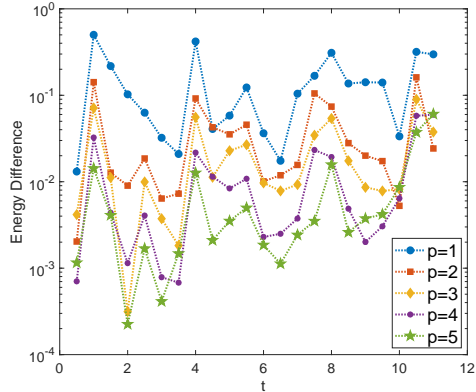}
   \caption{Energy difference between $p=6$ and lower $p$, as a function of time, Example \ref{example1p}.}\label{icosdifference}
\end{figure}\

 \section{Conclusions}\label{secconc}

In this work we initiate the study of $p$- and $hp$-version boundary elements for the wave equation. The analysis and numerical experiments show the efficient approximation of both smooth solutions and  geometric singularities in polyhedral domains, with the same convergence rates as known for $p$- and $hp$-approximations of time independent problems \cite{bh08, schwabs}. \\

For singular solutions the quasi-optimal $hp$-explicit estimates in this article complement the recent analysis of low-order approximations on algebraically graded meshes, for both finite and boundary element methods \cite{graded, schwabm}. In both cases the convergence is determined by the singularities of the solution at non-smooth boundary points of the domain. The analysis combines the time independent approximation results \cite{bh08} with the work by Plamenevskii and co-authors on the leading singular terms in the time dependent problem \cite{plamenevskii}. For screen problems the energy error $O(p^{-1})$ of the $p$-version has the same convergence rate as for an $h$-version on a $2$-graded mesh. For open polyhedral domains the solutions are less singular, and accordingly higher convergence rates are obtained. Numerical experiments illustrate these  on the icosahedron.\\


\appendix
\section{Asymptotic expansion for the solution}

In the following, let us describe the approach by Plamenevskii and coauthors (given in \cite{kokotov3}) to prove the asymptotic expansion of the solution to the wave equation near a singular point of the domain. For ease of comparison with the work of Plamenevskii, this Appendix adopts some of the notation from the analysis community e.g. the  $\sigma>0$ from the main body of the article is here called $\gamma$.\\

Consider the Dirichlet problem in the infinite cylinder $Q=K\times \mathbb{R}$
 \begin{align}\label{eq:dirplamenvskiiwave}
   \partial_t^{2}u - \Delta_{x}  u = f \ \text{ in } K \times \mathbb{R}, \nonumber \\
   u|_{\partial K \times \mathbb{R}} = 0 ,
 \end{align}
 where $K$ is an open cone in $\mathbb{R}_{x}^{n}$, $\Omega=K \cap S^{n-1}$ and let the boundary $\partial \Omega$ be smooth.
Applying the Fourier transform $\mathcal{F}_{t \rightarrow \tau}$, with $\tau=\sigma-i\gamma$, $\sigma \in \mathbb{R}$, $\gamma > 0$ to \eqref{eq:dirplamenvskiiwave} gives 
\begin{align}\label{eq:dirfourierplamenvskii}
 (-\Delta_{x} - \tau^{2}) u &=g \ \text{ in } K  \ , \ \ g(\cdot) = \mathcal{F} f(\tau,\cdot) =: \hat{f}(\tau,\cdot) ,\nonumber \\
 u|_{\partial K} &= 0.
\end{align}
Let $A(\tau)$ denote the closure in $L_{2}(K)$ of the operator $-\Delta_{x}-\tau^{2}$ which is originally defined for functions $\chi(|y|)|y|^{i \lambda_{k}} \Phi_{k}(\frac{y}{|y|})$, $k=1,2,\ldots $, where $\chi \in C_{0}^{\infty}(\overline{\mathbb{R}}_{+})$ is a cut-off function equal to one near the origin and $y \in K$. Here $\Phi_{k}$ are the eigenfunctions of the pencil $\mathcal{A}_{D}$ introduced with $i \lambda = \nu$ and $d=0$, $n=3$ in Section \ref{sectasympt}. 

\ \\
As shown in \cite{plamenevskii} for any $g \in L_{2}(K)$ and $\tau=\sigma-i \gamma$, $\gamma>0$, there exists a unique solution of $A(\tau) u = g$ satisfying
\begin{equation}\label{eq:energyestimate} 
 \gamma^{2} (|\tau|^{2} \lVert u \rVert_{L_{2}(K)}^{2} + \lVert \nabla u \rVert_{L_{2}(K)}^{2} ) \leq c \lVert g \rVert_{L_{2}(K)}^{2}
\end{equation}
with a constant $c$ independent of $g$ and $\tau$. If $\lambda$ is an eigenvalue of $\mathcal{A}_{D}$, then the homogeneous problem \eqref{eq:dirfourierplamenvskii} has a solution
\begin{equation}\label{eq:solutionhomogwave}
 r^{i \lambda} \sum_{k \geq 0} (\tau r)^{2 k} \Psi_{k}(\omega) \ ; \quad r=|x| \ , \ \omega=\frac{x}{|x|}
\end{equation}
where $\Psi_{0} = \Phi_{j}$ for $\lambda=\lambda_{j}$. Denote the series \eqref{eq:solutionhomogwave} by $ w_{j}$ if $\lambda=\lambda_{j}$, $j>0$, and by $\tilde{w}_{- j}(x,\tau)$ if $\lambda=\lambda_{-j}$, $j>0$ and by $w_{-}^{N}(x,\tau) (\tilde{w}_{-j}^{N}(x,\tau))$ their $N$-th partial sum.

\ \\
Let $\zeta \in C_{0}^{\infty}(\mathbb{R}^{3})$ and $\zeta=1$ near the origin $0$ and let $M\in \mathbb{R}$ be large such that
\begin{equation*}
 \tilde{g}:= (-\Delta_{x} - \tau^{2})(\zeta \tilde{w}_{-j}^{M}) \in L_{2}(K) .
\end{equation*}
Then the problem \eqref{eq:dirfourierplamenvskii} with $g=\tilde{g}$ and $u|_{\partial K} = 0$ has a solution $u$.
Setting $w_{j}:=\zeta \tilde{w}_{j}^{M} - u$, one observes that $w_{-j}$ depends neither on the cut-off function $\zeta$ nor on the number $M$ and $w_{-j}$ solves the homogeneous problem \eqref{eq:dirfourierplamenvskii} and has the asymptotic expansion $\tilde{w}_{-j}^{M}(x,\tau)$ near 0.
Replace $\tau$ by $\overline{\tau}$ in \eqref{eq:dirfourierplamenvskii} and denote by $w_{-j}(x,\overline{\tau})$ a corresponding solution with asymptotics $\tilde{w}_{-j}^{M}(x,\tau)$.

\ \\
For $\beta \leq 1$, $\chi \in C_{0}^{\infty}(\overline{\mathbb{R}}_{+})$, $\chi=1$ near the origin, define the space $DH_{\beta}(K,|\tau|)$ as completion of $C_{0}^{\infty}(\overline{K}\backslash 0)$ with respect to the norm
\begin{equation*}
  \lVert v \rVert_{DH_{\beta}(K,|\tau|)} := (\lVert \chi_{|\tau|} v \rVert_{H_{\beta}^{2}(K,|\tau|)}^{2}+\gamma^{2} \lVert v \rVert_{H_{\beta}^{1}(K,|\tau|)})^{1/2} 
\end{equation*}
where $\chi_{|\tau|}(y) = \chi(|\tau| |y|)$. Furthermore introduce
\begin{equation*}
    \lVert v \rVert_{RH_{\beta}(K,|\tau|)} := (\lVert f \rVert_{H_{\beta}^{0}(K)}^{2}+ \frac{|\tau|^{2-2 \beta}}{\gamma^{2}} \lVert f \rVert_{L_{2}(K)})^{1/2} .
\end{equation*}
Here we have used (see \cite{plamenevskii}) for $s>0$ integer, $\beta \in \mathbb{R}$, the space $H_{\beta}^{s}(K)$ being the completion of the set $C_{0}^{\infty}(\overline{K}\backslash 0)$ with respect to the norm
\begin{equation*}
   \lVert u \rVert_{H_{\beta}^{s}(K)} = (\sum_{|\alpha|\leq s} \int_{K} |y|^{2 (\beta-s+|\alpha|)} |D_{y}^{\alpha} u(y)|^{2} dy)^{1/2} ,
\end{equation*}
and for $q>0$, the space $H_{\beta}^{s}(K,q)$ with the norm
\begin{equation*}
   \lVert u \rVert_{H_{\beta}^{s}(K,q)} = (\sum_{k=0}^{s} q^{2k} \lVert u \rVert_{H_{\beta}^{s-k}(K)}^{2})^{1/2} .
\end{equation*}
Denote by $DV_{\beta}(K\times\mathbb{R}; \gamma)$ the space with the norm
\begin{equation*}
 \lVert u \rVert_{DV_{\beta}(K\times \mathbb{R} ; \gamma)} = ( \int_{\Im \tau = - \gamma} \lVert \mathcal{F}_{t \rightarrow \tau} u(\cdot, t)\rVert_{DH_{\beta}(K,|\tau|)}^{2} d\tau)^{1/2} ,
\end{equation*}
and by $RV_{\beta}(K\times\mathbb{R}; \gamma)$ the space with the norm
\begin{equation*}
 \lVert f \rVert_{RV_{\beta}(K\times \mathbb{R} ; \gamma)} = ( \int_{\Im \tau = - \gamma} \lVert \mathcal{F}_{t \rightarrow \tau} f(\cdot, t)\rVert_{RH_{\beta}(K,|\tau|)}^{2} d\tau)^{1/2} .
\end{equation*}
Also one introduces the operators
\begin{equation*}
 (Xu)(y,t) := \int_{\Im \tau = - \gamma} e^{i t \tau} \chi(|\tau| |y|) u(\tau) d\tau
\end{equation*}
and
\begin{equation*}
 (\Lambda f)(y,t) = \mathcal{F}^{-1}_{\tau \rightarrow t} |\tau| \mathcal{F}_{t' \rightarrow \tau} f(y,t') .
\end{equation*}
One sets 
\begin{equation*}
  U_{j}^{L_{j}}(\nu, \omega_{j}) = \sum_{k=0}^{L_{j}-1} \nu^{2k} \Psi_{k}(\omega),
\end{equation*}
where $\Psi_{k}$ as in \eqref{eq:solutionhomogwave} with $\lambda=\lambda_{j}$ and $L_{j}$ large enough. Then  there holds
\begin{theorem}[\cite{plamenevskii}]\label{app:theorem1}
 Let $\Lambda f \in RV_{\beta}(K\times \mathbb{R} ; \gamma)$ for $\beta \in (\beta_{k-j} , \beta_{k})$ with some $k=1,2,\ldots $. Then  the solution of \eqref{eq:dirplamenvskiiwave} has a representation
 \begin{equation}\label{eq:thm1representationofsol}
  u(y,t) = \sum_{j \in J} r^{i \lambda_{j}} U_{j}^{L_{j}}(r \partial_{t} , \omega) (X \check{c}_{j})(y,t) + \check{h}(y,t)
 \end{equation}
where $ \check{c}_{j}(t) = \mathcal{F}_{\tau \rightarrow t}^{-1}(\mathcal{F}_{t\rightarrow \tau} f(\cdot, t), w_{-j}(\cdot,\overline{\tau}))_{L_{2}(K)}$ and the remainder $\check{h}$ is subject to the estimate
\begin{equation}\label{eq:thm1remainderestimate}
 \gamma \lVert \check{h} \rVert_{DV_{\beta}(K\times\mathbb{R}, \gamma)} \leq c \lVert \Lambda f \rVert_{RV_{\beta}(K\times\mathbb{R}, \gamma)}
\end{equation}
with a constant $c$ independent of $\gamma> 0$.
\end{theorem}
We now follow again \cite{kokotov3} and  seek a solution of \eqref{eq:dirfourierplamenvskii} with $g\equiv 0$ and $n=3, d=0$ such that $ u \sim r^{i \lambda_{-k}} \Phi_{k}(\omega)$ as $r\rightarrow 0$ and $u \in L_{2}(K\backslash B_{\epsilon})$ where $B_{\epsilon} = \lbrace  x \in \mathbb{R}^{3}: |x| < \epsilon \rbrace $. First set
\begin{equation*}
  u(r,\omega,\tau) = r^{i \lambda_{-k}} \rho(r \tau) \Phi_{k}(\omega)
\end{equation*}
with $\rho$ to be defined later. Using
\begin{equation*}
 \lbrace (r\partial_{r})^{2} + r \partial_{r} - \Delta_{S} \rbrace r^{i \lambda_{-k}} \Phi_{k}(\omega) = 0,
\end{equation*}
with the Laplace operator $\Delta_{S}$ on $S^{2}$ we have 
\begin{equation*}
 (\tau r)^{2} \rho(r \tau) + (2 i \lambda_{-k} + 1)(r\tau) \rho'(r\tau) + (r \tau)^{2} \rho''(r\tau) = 0 .
\end{equation*}
Denoting $x=r \tau$, $\rho(x) = x^{\nu}\xi(x)$, where $\nu=(2-2i \lambda_{-k}-3)/2$ and $\Theta(y):= \xi(-iy)$
we get the Bessel equation
\begin{equation*}
 y^{2} \Theta''(y) + y \Theta'(y) - (\nu^{2}+y^{2})\Theta(y) = 0
\end{equation*}
and take as $\Theta$ the modified Bessel function $K_{\nu}$ of third kind yielding
\begin{equation*}
 \rho(t \tau) = c \  (r \tau)^{\nu}  K_{\nu}(i r \tau),
\end{equation*}
where the constant $c$ is given (via the condition $\rho(0) = 1$) as 
\begin{equation*}
  c = \pi^{-1}  \sin(\pi \nu) \Gamma(1-\nu) i^{\nu} 2^{1-\nu}.
\end{equation*}
Note that $u \in L_{2}(K \backslash B_{\epsilon})$ due to
\begin{equation}\label{eq:asymptotics}
 K_{\nu}(z)= \sqrt{\frac{\pi}{2 z}} e^{-z} \big[ \sum_{m=0}^{M-1} c(\nu,m) (2 z)^{-m} + O(|z|^{-M})\big] , |z| \rightarrow \infty \ .
\end{equation}
Further note that $e^{-i r \tau}$ decreases rapidly as $r\rightarrow + \infty$, $\tau = \sigma -i \gamma$, $\gamma> 0$. Hence 
\begin{equation}\label{eq:expressionofomega}
 w_{-k}(x,\tau) = r^{i \lambda_{-k}} \Phi_{k}(\omega) \frac{2^{1-\nu}}{\Gamma(\nu)} (i r \tau)^{\nu} K_{\nu}(i \tau r) , \ 2 \nu = \sqrt{1+4 \mu_{k}} .
\end{equation}
Applying the inverse Fourier transform gives
\begin{align}\label{eq:formofsolution}
 W_{-k}(x,t) &\!:=\!\! \mathcal{F}_{\tau \rightarrow t}^{-1} w_{-k}(x,\tau)\! =\!\! \frac{2^{1-N}}{\Gamma(\nu)\Gamma(\mu+1/2)} r^{\nu-\mu+i \lambda_{-k}} \Phi_{k}(\omega) (\frac{\partial}{\partial t})^{N} I_{N}(r,t,\mu,\nu) \nonumber \\ & =: (\frac{\partial}{\partial t})^{N} P_{N,k}(x,t),
\end{align}
where $r=|x|$, $N=[\nu]+m$, $\mu:=[\nu]-\nu+m$, $m>0$ arbitrary integer,
\begin{equation}\label{eq:defofIN}
 I_{N}(r,t,\mu,\nu) = \begin{cases} 
                          0 , & t < r \\
                          \pi^{1/2} (t^{2}-r^{2})^{\frac{2 \mu-1}{2}} F(\frac{\mu-\nu}{2},\frac{\mu+\nu}{2},\mu+1/2,1-\frac{t^{2}}{r^{2}}) , & t > r,
 \end{cases}
\end{equation}
and $F(a,b,c,z)$ is the hypergeometric function.

\begin{remark}
 From \eqref{eq:defofIN} follows that $W_{-k}(x,t) = 0$ if $ t < |x|$ and $\mathrm{sing}$ $ \mathrm{supp} W_{-k} \subset \lbrace (x,t) : |x|=t \rbrace $. The function $P_{N,k}(x,t)$ satisfies the homogeneous wave equation in $K \times \mathbb{R}_{+}$.
\end{remark}
Next, we look for a solution $u$ (of the homogeneous problem \eqref{eq:dirfourierplamenvskii}) with asymptotics
$r^{i \lambda_{k}} \Phi_{k}(\omega)$ as $x \rightarrow 0$, that is $u(x,\tau)=r^{i \lambda_{k}} \Phi_{k}(\omega) \rho(r \tau)$ with $\rho(0)=1$. Similarly to above one obtains 
\begin{equation*}
 u(x,t) = 2^{\nu} \Gamma(1+\nu) (i r \tau)^{-\nu} I_{\nu}(i r \tau) r^{i \lambda_{k}} \Phi_{k}(\omega),
\end{equation*}
where $ 2 \nu = \sqrt{1+ 4 \mu_{k}}$ and $I_{\nu}$ is the modified Bessel function \eqref{eq:defofIN}. Thus $w_{k}= u$ is 
\begin{equation*}
 w_{k}(x,\tau) = 2^{\nu} \Gamma(1+\nu) r^{i \lambda_{k}} \Phi_{k}(\omega) \sum_{m=0}^{\infty} \frac{(ir t)^{2m}}{m! \Gamma(m+\nu+1)}.
\end{equation*}
Next we consider the Dirichlet problem \eqref{eq:dirplamenvskiiwave} with inhomogeneous initial conditions, 
\begin{equation}\label{eq:nonhominitcond}
 u(0,x) = \phi(x) , \ \ \partial_{t} u(0,x) = \psi(x) .
\end{equation}
\begin{theorem}[\cite{kokotov3}]\label{app:theorem2}
   Let $k \in \mathbb{N}$, $\beta \in (\beta_{k+1},\beta_{k})$, $\gamma>0$ and $2 \nu_{j}=\sqrt{1+4\mu_{j}}$, $N_{j}=[\nu_{j}]+m$, $m\geq 4 $ integer. Let $N_{j}$ even, $N_{j}=2 l_{j}$. Set
   \begin{equation}\label{thmeq:formofcheckc}
     \check{c}_{j}(t) = \int_{K} \Delta^{l_{j}} \psi(y) P_{N_{j},j}(y,t) dy + \int_{K} \Delta^{l_{j}} \phi(y) \partial_{t} P_{N_{j},j}(y,t) dy.
   \end{equation}
with $P_{N,k}$ as in \eqref{eq:formofsolution}. Then there holds for the solution of \eqref{eq:dirplamenvskiiwave} satisfying \eqref{eq:nonhominitcond}
\begin{equation}\label{app:thm2aussage}
 u(x,t) = \chi(r) \sum_{j \in J} 2^{\nu_{j}} \Gamma(1+\nu_{j}) \lbrace \sum_{m=0}^{L_{j}} \frac{(r \partial_{t})^{2m} \check{c}_{j}(t)}{m! \Gamma(m+\nu_{j}+1)} \rbrace \Phi_{j}(\omega) r^{i\lambda_{j}} + \tilde{\rho}(x,t) .
\end{equation}
Here $\chi$ is a cut-off function with $\chi \equiv 1$ near $0$, $L_{j}$ are sufficiently large integers. $J:=\lbrace j: \Im \lambda_{j} \geq \beta_{k} -1/2 \rbrace$ and $\tilde{\rho}$ satisfies
\begin{equation*}
 \lVert \tilde{\rho} \rVert_{DV_{\beta}(K \times \mathbb{R},\gamma)} \leq c(\gamma)
\end{equation*}
\end{theorem}
\begin{remark}
 Analogous results for the Neumann problem of the wave equation are derived in \cite{kokotov}.
\end{remark}
\begin{remark}
 Since $\mathrm{supp} P_{N,j}\subset \lbrace (t,y) : t \geq |y| \rbrace$, we have $\check{c}_{j}(t) = 0$ if $t< \inf\lbrace |x|: x \in \mathrm{supp} \psi \cup \mathrm{supp} \phi \rbrace$.
 If $t>\sup\lbrace |x|: x \in \mathrm{supp} \psi \cup \mathrm{supp} \phi \rbrace$ then integration by parts in \eqref{thmeq:formofcheckc} gives
 \begin{equation*}
  \check{c}_{j}(\cdot) = \int_{K} \psi(y) W_{-j}(y,t) dy + \int_{K} \phi(y) \partial_{t} W_{-j}(y,t) dy.
 \end{equation*}
\end{remark}
Now, let us first consider problem \eqref{eq:dirplamenvskiiwave} in the infinite cylinder $K\times \mathbb{R}$. With the assumption of Theorem \ref{app:theorem1} for the right hand side $f$ the coefficient $\check{c}_{j}(\cdot)$ in \eqref{eq:thm1representationofsol} belongs to the Sobolev space $H^{3/2-\Im \lambda_{j}-\beta}(\mathbb{R})$. On the other hand
\begin{equation*}
 \check{c}_{j}(\cdot) = \int f(x,s) W_{-j}(x,\cdot-s) dx  ds 
\end{equation*}
belongs to the class $C^{\infty}(\alpha,+\infty)$ for any $\alpha > \sup\lbrace |x|+s: (x,s) \in \mathrm{sing} \ \mathrm{ supp} f \rbrace$. If $t < \inf\lbrace |x|+s: (x,s) \in \mathrm{supp} f \rbrace$, then we have $ \check{c}_{j}(t)=0$.
\begin{proof} of Theorem \ref{app:theorem2} \\
 Let $u$ solve \eqref{eq:dirplamenvskiiwave} with \eqref{eq:nonhominitcond} and consider 
 \begin{align}
   & w \in C_{0}^{\infty}(K \times \mathbb{R}) \label{thm2prooftool1}  \\
   w(x,0)=\phi(x) , \ \ & \frac{\partial w}{\partial t}(x,0) = \psi(x)  \label{thm2prooftool2},\\
   ( \partial_t^{2}w - \Delta_{x} w)_{+} &= \Theta_{+} (\partial_t^{2}w - \Delta_{x} w) ,  \label{thm2prooftool3} 
 \end{align}
 with  the characteristic function $\Theta_{+}$ of the semi axis  $\lbrace t: t \geq 0 \rbrace $.
Note that \eqref{thm2prooftool2}, \eqref{thm2prooftool3} are equivalent to
\begin{equation*}
 \partial_{t}^{2n+1} w(x,0) = \Delta_{x}^{n} \psi(x), \  \partial_{t}^{2n} w(x,0) = \Delta_{x}^{n} \phi(x) , \ n=0,1,2, \ldots
\end{equation*}
Note further that $v:=u-w$ satisfies
\begin{align}\label{thm2prooftool4}
 \partial_{t}^{2}v-\Delta_{x} v = - (\partial_{t}^{2}w-\Delta_{x}w) =: g \text{ in } K \times (0,\infty) \nonumber \\
 v|_{\partial K \times (0,\infty)} = 0 \\
 v|_{t=0} = 0 , (\partial_{t} v)|_{t=0} = 0 \nonumber
\end{align}
Consider in the infinite cylinder $K \times \mathbb{R}$:
\begin{align}\label{thm2prooftool5}
 \partial_{t}^{2}v-\Delta_{x}v = - (\partial_{t}^{2}w-\Delta_{x}w)_{+} \text{ in } K \times (0,\infty) \nonumber \\
 v|_{\partial K \times (0,\infty)} = 0
\end{align}
First by applying a priori estimates in weighted spaces \eqref{eq:propapriori} from Proposition \ref{app:prop1} and the Paley-Wiener theorem, we deduce that $v$ 
is smooth in $t$, $v\equiv 0$ for $t<0$. That means $v$ coincides for $t>0$ with the solution of \eqref{thm2prooftool4}.

\ \\
Next, we observe that $ (\mathrm{supp} w)|_{\mathbb{R}_{\times}^{n}} \cap 0 = \emptyset$ where $0$ is the vertex of $K$. Therefore the asymptotics of $u$ and $v$ near $0$ coincide. Let $g:= - (\partial_{t}^{2}-\Delta_{x}) w$ and $g_{+} = \Theta_{+} g$. According to Theorem \ref{app:theorem1} there holds
\begin{equation}\label{eq:proofstepbthm3}
 v(x,t) = \sum_{j\in J} r^{i \lambda_{j}} U_{j}^{L_{j}}(r \partial_{t}), \omega) (X \check{c}_{j})(x,t) + \check{h}(x,t)
\end{equation}
where $\check{c}_{j}(t) = \mathcal{F}_{\tau \rightarrow t}^{-1} (\hat{g}_{+}(\cdot,\tau), w_{-j}(\cdot,\overline{\tau}))_{L_{2}(K)}$ with $\check{h}$ satisfying \eqref{eq:thm1remainderestimate} where $f$ is replaced by $g_{+}$. 

\ \\
Next, let us express $\check{c}_{j}(t)$ in terms of data of the homogeneous problem \eqref{eq:dirplamenvskiiwave} with inhomogeneous initial conditions \eqref{eq:nonhominitcond}. We have
\begin{align*}
 \check{c}_{j}(t) &= \int\limits_{K} \!\! dy \int\limits_{-\infty}^{+\infty} g_{+}(y,s) W_{-j}(y,t\!-\!s) ds = \int\limits_{K} \!\! dy \int\limits_{0}^{\infty} g(y,s) (\frac{\partial}{\partial s})^{N} P_{N_{j},j}(y,t\!-\!s) ds \\
 &= (-1)^{N_{j}} \int_{K} \!\! dy \int_{0}^{\infty} \partial_{s}^{N_{j}} g(y,s) P_{N_{j},j}(y,t-s) ds \\
 &= (-1)^{N_{j}} \int_{K} dy \int_{0}^{\infty} - (\partial_{s}^{2}-\Delta_{y}) \partial_{s}^{N_{j}} w(y,s) P_{N_{j},j}(y,t-s) ds \\
 &= (-1)^{N_{j}}  \int\limits_{K} \!\! dy \lbrace \partial_{s}^{N_{j}+1}w(y,0) P_{N_{j},j}(y,t)\!-\!\!\int\limits_{0}^{\infty} \! \partial_{s}^{N_{j}+1} w(y,s) \partial_{s} P_{N_{j},j}(y,t\!-\!s) ds  \\ &+ \int\limits_{0}^{\infty} \Delta_{y} \partial_{s}^{N_{j}} w(y,s) P_{N_{j},j}(y,t-s) ds \rbrace \\
 &= (-1)^{N_{j}} \int\limits_{K} dy \lbrace \partial_{s}^{N_{j}+1}w(y,0) P_{N_{j},j}(y,t)+ \partial_{s}^{N_{j}}w(y,0) \partial_{s}P_{N_{j},j}(y,t)\\ &+\int\limits_{0}^{\infty} \partial_{s}^{N_{j}} w(y,s) \partial_{s}^{2} P_{N_{j},j}(y,t-s) ds + \int\limits_{0}^{\infty} \partial_{s}^{N_{j}} w(y,s) \Delta_{y} P_{N_{j},j}(y,t-s) ds \rbrace
\end{align*}
Note that $ (\partial_{s}^{2}-\Delta_{y}) P_{N_{j},j} = 0 $.
Hence
\begin{equation*}
 \check{c}_{j}(t) = (-1)^{N_{j}} (\int_{K} \partial_{t}^{N_{j}+1} w(y,0) P_{N_{j},j}(y,t) dy + \int_{K} \partial_{t}^{N_{j}} w(y,0) \partial_{t} P_{N_{j},j}(y,t) dy )
\end{equation*}
Setting $N_{j}=2 l_{j}$ gives \eqref{thmeq:formofcheckc}.

\ \\
It remains to show that one can omit the operator $X$. One observes 
\begin{equation}\label{thm2prooftool6}
 X \check{c}_{j}(x,t) - \chi(r) \check{c}_{j}(t) = \int_{\Im \tau = - \gamma} e^{i t \tau} (\chi(|\tau|r)-\chi(r)) c_{j}(\tau) d\tau
\end{equation}
where $c_{j}(\tau) = (\hat{g}_{+}(\cdot,\tau),w_{-j}(\cdot,\overline{\tau}))_{L_{2}(K)}$. Using the inclusion $g_{+} \in C_{0}^{\infty}(K\times \mathbb{R})$, the explicit expression \eqref{eq:expressionofomega} for $w_{-j}(x,\tau)$ and the asymptotics \eqref{eq:asymptotics}, one gets
\begin{equation}\label{thm2prooftool7}
 |(\frac{d}{d \tau})^{k} c_{j}(\tau)| \leq c(\gamma,k,N) |\tau|^{-N} \ \ , \ \forall k, N \geq 0
\end{equation}
It follows from \eqref{thm2prooftool6} and \eqref{thm2prooftool7} that
\begin{equation*}
 \kappa(x,t) := \sum_{j \in J} r^{i \lambda_{j}} U_{j}^{L_{j}}(r \partial_t,\omega) (X \hat c_{j}(x,t)) - \chi(r) \sum_{j \in J} r^{i \lambda_{j}} U_{j}^{L_{j}}(r \partial_t, \omega) \check{c}_{j}(t)
\end{equation*}
belongs to $V_{\beta'}^{s}(K \times \mathbb{R},\gamma)$ for any $s \in \mathbb{N}_{0}$, $\beta' \in \mathbb{R}$. Hence $\lVert \kappa \rVert_{DV_{\beta}(K \times \mathbb{R},\gamma)} $ is finite. Setting $\tilde{\rho}(x,t) = \kappa(x,t) + \check{h}(x,t)$ completes the proof of Theorem \ref{app:theorem2}.
\end{proof}

To justify the asymptotic formula \eqref{app:thm2aussage} for the solution 
Kokotov and Plamenevskii study in \cite{kokotov2} the solvability of \eqref{eq:dirplamenvskiiwave} (with Neumann conditions) in a scale of weighted Sobolev spaces. The method is based on " combined " estimates for the solution as follows. The operator $(-\Delta_{x}-\tau^{2}) $ in \eqref{eq:dirfourierplamenvskii}
$\tau=\sigma-i \gamma$, $\sigma \in \mathbb{R}$, $\gamma >0$, is elliptic for fixed parameter 
$\tau$, but hyperbolic in $\tau$. Now one has to estimate the solution uniformly with respect to the parameter. Therefore the cone $K$ is divided into various zones:
\begin{itemize}
 \item near the vertex 
 where one uses the weighted elliptic estimate \eqref{eq:propaposterioriestimate},
 \item far from the vertex 
 where one uses the weighted hyperbolic estimate \eqref{propproofhelp2},
 \item in the intermediate zone where one uses the weak global estimate \eqref{eq:energypropestimate}, which holds in the entire cone and follows from
 \begin{equation*}
  \gamma^{2} \int_{Q} e^{-2 \gamma t} |\nabla_{x,t} u(x,t)|^{2} dx dt \leq c \int_{Q} e^{-2 \gamma t} | (\partial_{t}^{2}-\Delta_{x}) u(x,t)|^{2} dx dt
 \end{equation*}

In this way Kokotov and Plamenevskii \cite{kokotov4} obtain a combined a priori estimate for the solution in a scale of weighted spaces ($n=3,d=0$).
\end{itemize}

\begin{proposition}[\cite{kokotov4}, Proposition 2.8]\label{app:prop1}
 Let $\beta \leq 1$, $\beta \neq 1-\frac{1}{2} \sqrt{1+ 4 \mu_{k}}$ and $q \in \mathbb{N}$. The solution $u$ of \eqref{eq:dirfourierplamenvskii} satisfies the a priori estimate
 \begin{flalign}\label{eq:propapriori}
  \lVert \chi(\tau) u \rVert_{H_{\beta+q}^{2+q}(K,|\tau|)}^{2} + \gamma^{2} \lVert u \rVert_{H_{\beta+q}^{1+q}(K,|\tau|)}^{2} \leq c \left(\sum_{j=0}^{q} \Big(\frac{|\tau|}{\gamma}\Big)^{2 j} \lVert g \rVert_{H_{\beta+q-j}^{q-j}(K,|\tau|)}^{2}   + \Big(\frac{|\tau|^{1-\beta+q}}{\gamma^{1+q}}\Big)^{2} \lVert g \rVert_{L_{2}(K)}^{2} \right) 
 \end{flalign}
with a constant $c$ which is independent of $\tau=\sigma-i \gamma , \gamma> 0$.
\end{proposition}
In the following we sketch the proof of the above proposition (see also \cite{kokotov}). After the change of variable $\eta=p x$ with $p=|\tau|$ the problem \eqref{eq:dirfourierplamenvskii} takes the form 
\begin{flalign*}
 & L(D_{\eta},\Theta) U(\eta,\tau) = F(\eta,\tau) , \ \ \eta \in K & \\
 & U(\eta,\tau) = 0 , \ \ \eta \in \partial K
\end{flalign*}
where $U(\eta,\tau) = u(p^{-1}\eta,\tau)$, $F(\eta,\tau) = p^{-2} \hat{f}(p^{-1}\eta,\tau)$ with $\hat{f}(x,\tau) = g(x,\tau)$, $\Theta= \Theta(\tau) = \tau p^{-1} $, $\tau=\sigma - i \gamma$, $\sigma \in \mathbb{R}$, $\gamma>0$ and $L(D_{x},\tau) = - \Delta_{x}-\tau^{2}$. 

\ \\
Note that for $v(x)=\chi(x) |x|^{i \lambda_{k}} \Phi_{k}(\omega)$ there holds
\begin{equation}\label{eq:energypropestimate}
 \gamma^{2} \int_{K} (p^{2} |v(x)|^{2} + \nabla_{x} v(x))^{2} dx \leq c \int_{K} | (-\Delta_{x} - (\sigma-i\gamma)) v(x)|^{2} dx
\end{equation}
and
\begin{equation*}
 \gamma^{2} \lVert v \rVert_{H_{0}^{1}(K;p)}^{2} \leq c \lVert (-\Delta_{x}-\tau^{2}) v \rVert_{L_{2}(K)}^{2}
\end{equation*}
with a constant $c$, independent of $\tau=\sigma-i \gamma$, $\sigma \in \mathbb{R}, \gamma> 0$.

\ \\
Next, we take $\beta \leq 1$ and assume that the line $\Im \lambda = \beta-1/2$ does not contain points of the spectrum of the pencil $\mathcal{A}(\lambda)$. Then there holds \cite[Proposition 1.3]{kokotov4}
\begin{equation}\label{propproofhelp1}
 \lVert \chi_{p} v \rVert_{H_{\beta}^{2}(K;p)}^{2} + \gamma^{2} \lVert v \rVert_{H_{\beta}^{1}(K;p)}^{2} \leq c \left\lbrace \lVert f \rVert_{H_{\beta}^{0}(K)}^{2} + (\frac{p^{1-\beta}}{\gamma})^{2} \lVert f \rVert_{L_{2}(K)}^{2}\right \rbrace,
\end{equation}
where $f = L(D_{x},\tau)v, \chi_{p}(x)=\chi(px), \chi(x) = \chi^{1}(|x|)$, $\chi^{1} \in C_{0}^{\infty}(\mathbb{R}), \chi^{1} \equiv 1$ near $0$, and $c$ is independent of $\tau=\sigma-i \gamma, \sigma \in \mathbb{R}, \gamma > 0$.

\ \\
Furthermore from \cite{kokotov2}, for any $\beta \in \mathbb{R}$ and $U \in H_{\beta}^{s+1}(K;1)$ such that $U = 0 $ on $\partial K$ there holds
\begin{equation}\label{propproofhelp2}
\left(\frac{\gamma}{p}\right)^{2} \lVert \kappa_{\infty} \rVert_{H_{\beta}^{s+1}(K;1)}^{2} \leq c \left\lbrace \lVert \psi_{\infty} L(D_{\eta},\Theta) U \rVert_{H_{\beta}^{s}(K;1)}^{2} + \lVert \psi_{\infty} U \rVert_{H_{\beta-1}^{s+1}(K;1)}^{2} \right\rbrace,
\end{equation}
where $\kappa_{\infty}(y) = \kappa_{\infty}^{1}(|y|)$, $\psi_{\infty}(y)=\psi_{\infty}^{1}(|y|)$, $\kappa_{\infty}^{1},\psi_{\infty}^{1} \in C^{\infty}(\mathbb{R})$, $\kappa_{\infty}^{1} \psi^{1}_{\infty} = \kappa_{\infty}^{1}$; $\kappa_{\infty}^{1}$ and $\psi_{\infty}^{1}$ vanish near $0$ and are equal to $1$ at infinity, $s=0,1,\ldots$ and the constant $c$ is independent of the parameters .

\ \\
We note that the estimate \eqref{eq:energypropestimate} is used in \cite{kokotov2} to show \eqref{propproofhelp2}.

\ \\
The following result by Kondratiev is crucial for the derivation of the a priori estimate \eqref{eq:propapriori}. 
\begin{proposition}[\cite{kondratiev}]
 Let $\chi,\psi \in C_{0}^{\infty}(\overline{K})$, $\chi = 1$ near the vertex $0$ of the cone $K$ and $\chi \psi=\chi$. Let the line $\Im \lambda = \beta-s-1+(n-d-2)/2$
do not contain eigenvalues of the pencil $\mathcal{A}$. then for all $U \in H_{\beta}^{s+2}(K;1)$ such that $U=0$ on $\partial K $ there holds
\begin{equation}\label{eq:propaposterioriestimate}
\lVert \chi U \rVert_{H_{\beta}^{s+2}(K;1)}^{2} \leq c \left\{ \lVert \psi L(D_{\eta},\Theta) U \rVert_{H_{\beta}^{s}(K;1)}^{2} + \lVert \psi U \rVert_{H_{\beta}^{s+1}(K;1)}^{2} \right\}.
\end{equation}
\end{proposition}

Proceeding by induction on $q$ and using \eqref{propproofhelp1} at the first step, from \eqref{propproofhelp2} and  \eqref{eq:propaposterioriestimate} (for $n=3,d=0$) one obtains the assertion \eqref{eq:propapriori} with $p=|\tau|$ and $v=u$.

\ \\
Next, we consider the asymptotics for the normal derivative of the solution of \eqref{eq:oriProblem2}. If one solves \eqref{eq:oriProblem2} with Dirichlet boundary conditions by a boundary integral equation, then one is interested in the unknown normal derivative $\partial_{n} u$ on $\Gamma$ rather than in $u$. Taking the normal derivative of the decomposition \eqref{eq:proofstepbthm3} for $v$ (in the proof of Theorem \ref{app:theorem2}) gives results for $\partial_{n} u$ on $\Gamma=\partial K$ for u in \eqref{eq:oriProblem2}. The situation for the Dirichlet problem of the Laplace operator on an infinite wedge $\Omega=\mathbb{R} \times K$ is analysed in Theorem 7 in \cite{petersdorff2} and in a polyhedral cone $\Omega$ in Theorem 8 in \cite{petersdorff2} with the limit case of a screen in Example 4 in \cite{petersdorff2}. Redoing the derivation of the singularity terms after Theorem \ref{app:theorem1} for the limit case of a screen with corresponding $\lambda_{-k}$ and $\Phi_{k}$ one obtains for a circular screen the expansion \eqref{decompositionEdge} and for a polygonal screen the expansion \eqref{decomposition}. Here one first must modify the decomposition \eqref{app:thm2aussage} of Theorem \ref{app:theorem2} for the solution $u$ of \eqref{eq:oriProblem2} in a polyhedral cone. As in \cite{kozlov,dauge} one uses a partition of unity near an edge of the cone together with another dyadic partition of unity along that edge and takes a right hand side $f$ in \eqref{eq:dirplamenvskiiwave} with $\mathrm{supp} f$ compact. This gives the extension of Theorem \ref{app:theorem1} to a polyhedral cone. Finally a refined analysis as the one given in \cite{petersdorff2} yields also for a polyhedral cone a tensor product decomposition like \eqref{app:thm2aussage} extending Theorem \ref{app:theorem2}. The Neumann problem in \eqref{eq:oriProblem2} can be treated analogously. For the result corresponding to Theorem \ref{app:theorem1} and Theorem \ref{app:theorem2} see \cite{kokotov2}. Then following \cite{petersdorff} for $u|_{\Gamma}$ one can adjust the above given procedure to the wave equation outside the screen and finally one obtains the decompositions \eqref{decompostionEdget} and \eqref{decompositiont} for $u|_{\Gamma}$ for circular and polygonal screens, respectively.

  \bibliographystyle{abbrv}

\end{document}